\documentclass[11pt]{article}

\usepackage[cp1252]{inputenc}
\usepackage[english]{babel}
\usepackage[a4paper]{geometry}
\usepackage{amsmath,amsfonts,amssymb,amsthm,cases,upgreek}
\usepackage{empheq}
\usepackage{stmaryrd}
\usepackage{dsfont}
\usepackage{graphicx}
\usepackage{comment}
\usepackage{bm} 

\newtheorem{thm}{Theorem}[section]
\newtheorem{definition}[thm]{Definition}
\newtheorem{lemma}[thm]{Lemma}

\newtheorem{prop}[thm]{Proposition}
\newtheorem{remark}[thm]{Remark}

\newcommand{\R}{\mathbb{R}}
\newcommand{\RR}{\mathbb{R}^{2}}
\newcommand{\RD}{\mathbb{R}^{d}}
\newcommand{\RDD}{\mathbb{R}^{d+1}}
\newcommand{\Hdot}{\dot{H}}

\newcommand{\Us}{\text{U}_{\! \sslash}}
\newcommand{\Umus}{\text{U}^{\mu}_{\! \sslash}}
\newcommand{\psit}{\widetilde{\psi}}

\newcommand{\Hb}{H_{b} \hspace{-0.05cm} \left(\text{div}^{\mu}_{0},\Omega \right)}
\newcommand{\psia}{\psi_{(\alpha)}}
\newcommand{\Umua}{\text{U}^{\mu}_{(\alpha)}}
\newcommand{\lver}{\left\lvert}
\newcommand{\llver}{\left\lvert\left\lvert}
\newcommand{\rver}{\right\rvert}
\newcommand{\rrver}{\right\rvert\right\rvert}

\renewcommand{\footnote}[1]{\textsuperscript{\addtocounter{footnote}{1}(\thefootnote)}\footnotetext{#1}}

\def \epsilon {\varepsilon}

\SetSymbolFont{stmry}{bold}{U}{stmry}{m}{n}

\begin{document}
\title{\textbf{Coriolis effect on water waves}}
\author{Benjamin MELINAND\footnote{IMB, Universit\'e de Bordeaux. Email : benjamin.melinand@math.u-bordeaux.fr}}
\date{October 2015}

\maketitle

\vspace{-1cm}

\begin{abstract}
\noindent This paper is devoted to the study of water waves under the influence of the gravity and the Coriolis force. It is quite common in the physical literature that the rotating shallow water equations are used to study such water waves.  We prove a local wellposedness theorem for the water waves equations with vorticity and Coriolis force, taking into account the dependence on various physical parameters and we justify rigorously the shallow water model. We also consider a possible non constant pressure at the surface that can be used to describe meteorological disturbances such as storms or pressure jumps for instance.
\end{abstract}

\section{Introduction}

\subsection{Presentation of the problem}

\noindent There has been a lot of interest on the Cauchy problem for the irrotational water waves problem since the work of S. Wu (\cite{Wu_2D} and \cite{Wu_3D}). More relevant for our present work is the Eulerian approach developed by D. Lannes (\cite{Lannes_wellposedness_ww}) in the presence of a bottom. Another program initiated by W. Craig (\cite{craig_boussinesq}) consists in justifying the use of the many asymptotic models that exist in the physical literature to describe the motion of water waves. This requires a local wellposedness result that is uniform with respect to the small parameters involved (typically, the shallow water parameter). This was achieved by B. Alvarez-Samaniego and D. Lannes (\cite{Alvarez_Lannes}) for many regimes; other references in this direction are (\cite{schneider_wayne_NLS}, \cite{schneider_wayne_longwave}, \cite{Iguchi_shallow_water}). The irrotational framework is however not always the relevant one. When dealing with wave-current interactions or, at larger scales, if one wants to take into account the Coriolis force. The latter configuration motivates the present study. Several authors considered the local wellposedness theory for the water waves equations in the presence of vorticity (\cite{coutand_shkoller_ww} or \cite{lindblad_ww} for instance). Recently, A. Castro and D. Lannes proposed a generalization of the Zakharov-Craig-Sulem formulation (see \cite{Zakharov}, \cite{Craig_Sulem_1}, \cite{Craig_Sulem_2}, \cite{abz_zaka_euler} for an explanation of this formulation), and gave a system of three equations that allow for the presence of vorticity. Then, they used it to derive new shallow water models that describe wave current interactions and more generally the coupling between waves and vorticity effects  (\cite{Castro_Lannes_vorticity} and \cite{Castro_Lannes_shallow_water}). In this paper, we base our study on their formulation.
\medskip

\noindent This paper is organized in three parts : firstly we derive a generalization of the Castro-Lannes formulation that takes into account the Coriolis forcing as well as non flat bottoms and a non constant pressure at the surface; secondly, we prove a local wellposedness result taking account the dependence of small parameters; Finally, we justify that the rotational shallow water model is a good asymptotic model of the rotational water waves equations under a Coriolis forcing. 
\medskip

\noindent We model the sea by an incompressible ideal fluid bounded from below by the seabed and from above by a free surface. We suppose that the seabed and the surface are graphs above the still water level. The pressure at the surface is of the form $P + P_{\text{ref}}$ where $P(t,\cdot)$ models a meteorological disturbance and $P_{\text{ref}}$ is a constant which represents the pressure far from the meteorological disturbance. We denote by $d$ the horizontal dimension, which is equal to $1$ or $2$. The horizontal variable is $X \in \RD$ and $z \in \mathbb{R}$ is the vertical variable. $H$ is the typical water depth. The water occupies the domain $\Omega_{t} := \{ (X,z) \in \RDD \text{ , } -H + b(X) < z < \zeta (t,X) \}$. The water is homogeneous (constant density $\rho$), inviscid with no surface tension. We denote by $\textbf{U}$ the velocity of the fluid, $\textbf{V}$ is the horizontal component of the velocity and $\textbf{w}$ its vertical component. The water is under the influence of the gravity $\bm{g} = -g \bm{e_{z}}$ and the rotation of the Earth with a rotation vector $\textbf{f} = \frac{f}{2} \bm{e_{z}}$. Finally, we define the pressure in the fluid domain by  $\mathcal{P}$. The equations governing the motion of the surface of an ideal fluid under the influence of gravity and Coriolis force are the free surface Euler Coriolis equations \footnote{We  consider that the centrifugal potential is constant and included in the pressure term.}

\begin{equation}\label{Euler_equations}
\left\{
\begin{aligned}
&\partial_{t} \textbf{U} + \left( \textbf{U} \cdot \nabla_{\! X,z} \right) \textbf{U} + \textbf{f} \times \textbf{U} = - \frac{1}{\rho} \nabla_{\! X,z} \mathcal{P} - g \bm{e_{z}} \text{ in } \Omega_{t},\\
&\text{div} \; \textbf{U} = 0 \text{ in } \Omega_{t},
\end{aligned}
\right.
\end{equation}

\noindent with the boundary conditions

\begin{equation}\label{bound_cond}
\left\{
\begin{aligned}
&\partial_{t} \zeta	- \underline{\text{U}} \cdot N = 0, \\
&\text{U}_{b} \cdot N_{b} = 0,
\end{aligned}
\right.
\end{equation}

\noindent where $N = \begin{pmatrix} - \nabla \zeta \\ 1 \end{pmatrix}$, $N_{b} = \begin{pmatrix} - \nabla b \\ 1 \end{pmatrix}$, $\underline{\text{U}} = \begin{pmatrix} \underline{\text{V}} \\ \underline{\text{w}} \end{pmatrix} = \textbf{U}_{|z=\zeta}$ and $\text{U}_{b} = \begin{pmatrix} \text{V}_{b} \\ \text{w}_{b} \end{pmatrix} = \textbf{U}_{|z=-H+b}.$ 

\medskip

\noindent The pressure $\mathcal{P}$ can be decomposed as the surface contribution and the internal pressure

\begin{equation*}
\mathcal{P}(t,X,z) = P(t,X) + P_{\text{ref}} + \widetilde{\mathcal{P}}(t,X,z), 
\end{equation*}

\noindent with $\widetilde{\mathcal{P}}_{|z=\zeta} = 0$.

\begin{remark}
\noindent In this paper, we identify functions on $\R^{2}$ as function on $\R^{3}$. Then, the gradient, the curl and the divergence operators become in the two dimensional case

\upshape
\begin{equation*}
\nabla_{\! X,z} f = \begin{pmatrix}  \partial_{x} f \\ 0 \\ \partial_{z} f \end{pmatrix} \text{, } \text{curl} \; \textbf{A} = \begin{pmatrix}  -\partial_{z} \textbf{A}_{2} \\ \partial_{z} \textbf{A}_{1} - \partial_{x} \textbf{A}_{3} \\ -\partial_{x} \textbf{A}_{2} \end{pmatrix} \text{, } \text{div} \; \textbf{A} = \partial_{x} \textbf{A}_{1} + \partial_{z} \textbf{A}_{3}.
\end{equation*}
\itshape
\end{remark}

\noindent In order to obtain some asymptotic models we nondimensionalize the previous equations. There are five important physical parameters : the typical amplitude of the surface $a$, the typical amplitude of the bathymetry $a_{\text{bott}}$, the typical horizontal scale $L$, the characteristic water depth $H$ and the typical Coriolis frequency $f$. Then we can introduce four dimensionless parameters 

\begin{equation}\label{parameters}
\epsilon =\frac{a}{H} \text{, } \beta = \frac{a_{\text{bott}}}{H} \text{, } \mu = \frac{H^{2}}{L^{2}} \text{ and } \text{Ro} = \frac{a}{f L} \sqrt{\frac{g}{H}}, 
\end{equation}

\noindent where $\epsilon$ is called the nonlinearity parameter, $\beta$ the bathymetric parameter, $\mu$ the shallowness parameter and $\text{Ro}$ the Rossby number. We also nondimensionalize the variables and the unknowns. We introduce (see \cite{Lannes_ww} and \cite{my_proud_res} for instance for an explanation of this nondimensionalization)

\begin{equation}\label{nondimens_variables}
 \left\{ 
 \begin{aligned}
 &X' = \frac{X}{L} \text{, } z' = \frac{z}{H} \text{, } \zeta' = \frac{\zeta}{a} \text{, } b' = \frac{b}{a_{\text{bott}}} \text{, } t' = \frac{\sqrt{gH}}{L} t \text{, }\\
 &\textbf{V}' = \sqrt{\frac{H}{g}} \frac{\textbf{V}}{a} \text{, } \textbf{w}'= H \sqrt{\frac{H}{g}} \frac{\textbf{w}}{aL} \text{, } P' = \frac{P}{\rho g a} \text{ and } \widetilde{\mathcal{P}}^{\prime} = \frac{\widetilde{\mathcal{P}}}{\rho g H}.
 \end{aligned}
 \right.
\end{equation}

\noindent In this paper, we use the following notations

\begin{equation}\label{nondim_op}
\nabla^{\mu}_{\! X',z'} = \begin{pmatrix} \sqrt{\mu} \nabla_{\! X'} \\ \partial_{z'} \end{pmatrix} \text{, } \text{curl}^{\mu} = \nabla^{\mu}_{\! X',z'} \times \text{, } \text{div}^{\mu} = \nabla^{\mu}_{\! X',z'} \cdot .
\end{equation}
 
\noindent We also define 

\begin{equation}\label{nondim_U}
\textbf{U}^{\mu} = \begin{pmatrix} \sqrt{\mu} \textbf{V}' \\ \textbf{w}' \end{pmatrix} \text{, } \bm{\omega}'=\frac{1}{\mu} \text{curl}^{\mu} \textbf{U}^{\mu} \text{, } \underline{\text{U}}^{\mu} = \textbf{U}^{\mu}_{|z'= \epsilon \zeta'} \text{, } \text{U}^{\mu}_{b} = \textbf{U}^{\mu}_{|z'=-1+ \beta b'},
\end{equation}

\noindent and 

\begin{equation}
 N^{\mu} =  \begin{pmatrix} - \epsilon \sqrt{\mu} \nabla \zeta' \\ 1 \end{pmatrix} \text{, } N_{\! b}^{\mu} = \begin{pmatrix} - \beta \sqrt{\mu} \nabla b' \\ 1 \end{pmatrix}.
\end{equation}

\noindent Notice that our nondimensionalization of the vorticity allows us to consider only weakly sheared flows (see \cite{Castro_Lannes_shallow_water}, \cite{teshukov_shear_flows}, \cite{richard_gravi_shear_flows}). The nondimensionalized fluid domain is

\begin{equation}
\Omega^{\prime}_{t'} := \{ (X',z') \in \RDD \text{ , } -1 + \beta b'(X') < z' < \epsilon \zeta' (t',X') \}.
\end{equation}

\noindent Finally, if $\textbf{V}  = \begin{pmatrix} \text{V}_{1} \\ \text{V}_{2} \end{pmatrix} \in \RR$, we define $\textbf{V}$ by $\textbf{V}^{\perp} = \begin{pmatrix} -\text{V}_{2} \\ \text{V}_{1} \end{pmatrix}$. Then, the Euler Coriolis equations \eqref{Euler_equations} become 

\small
\begin{equation}\label{nondim_Euler_equations}
\left\{
\begin{aligned}
&\partial_{t'} \textbf{U}^{\mu} + \frac{\epsilon}{\mu} \left( \textbf{U}^{\mu} \cdot \nabla_{\! X',z'}^{\mu} \right) \textbf{U}^{\mu} + \frac{\epsilon \sqrt{\mu}}{\text{Ro}} \begin{pmatrix} \;\; \textbf{V}^{\prime \perp} \\ 0 \end{pmatrix} = - \sqrt{\mu} \begin{pmatrix} \nabla \! P' \\ 0 \end{pmatrix} - \frac{1}{\epsilon} \nabla_{\! X',z'}^{\mu} \widetilde{\mathcal{P}}^{\prime} - \frac{1}{\epsilon} \bm{e_{z}} \text{ in } \Omega^{\prime}_{t},\\
&\text{div}^{\mu}_{\! X' \! ,z'} \; \textbf{U}^{\mu} = 0 \text{ in } \Omega^{\prime}_{t},
\end{aligned}
\right.
\end{equation}
\normalsize

\noindent with the boundary conditions

\begin{equation}\label{nondim_bound_cond}
\left\{
\begin{aligned}
&\partial_{t'} \zeta'	- \frac{1}{\mu} \underline{\text{U}}^{\mu} \cdot N^{\mu} = 0, \\
&\text{U}^{\mu}_{b} \cdot N_{b}^{\mu} = 0.
\end{aligned}
\right.
\end{equation}

\noindent In the following we omit the primes. In \cite{Castro_Lannes_vorticity}, A. Castro and D. Lannes derived a formulation of the water waves equations with vorticity. We outline the main ideas of this formulation and extend it to take into account the Coriolis force; even in absence of Coriolis forcing, our results extend the result of \cite{Castro_Lannes_vorticity} by allowing non flat bottoms. First, applying  the $\text{curl}^{\mu}$ operator to the first equation of \eqref{nondim_Euler_equations} we obtain an equation on $\bm{\omega}$

\begin{equation}\label{vorticity_eq}
\partial_{t} \bm{\omega} + \frac{\epsilon}{\mu} \left(\textbf{U}^{\mu} \cdot \nabla^{\mu}_{\! X,z} \right) \bm{\omega} = \frac{\epsilon}{\mu} \bm{\omega} \cdot \nabla^{\mu}_{\; X,z} \textbf{U}^{\mu} + \frac{\epsilon}{\mu \text{Ro}} \partial_{z} \textbf{U}^{\mu}.
\end{equation}

\noindent Furthermore, taking the trace at the surface of the first equation of \eqref{nondim_Euler_equations} we get

\begin{equation}\label{eq_int1}
\partial_{t} \underline{\text{U}}^{\mu} + \epsilon \left( \underline{\text{V}} \cdot \nabla_{\! X} \right) \underline{\text{U}}^{\mu} + \frac{\epsilon \sqrt{\mu}}{\text{Ro}} \begin{pmatrix} \;\;  \underline{\text{V}}^{\perp} \\ 0 \end{pmatrix} = - \sqrt{\mu} \begin{pmatrix} \nabla \! P \\ 0  \end{pmatrix} - \frac{1}{\epsilon} \begin{pmatrix} 0 \\ 1  \end{pmatrix} - \frac{1}{\epsilon} \left( \partial_{z} \widetilde{\mathcal{P}} \right)_{|z=\epsilon \zeta} N^{\mu}.
\end{equation}

\noindent Then, in order to eliminate the term $\left( \partial_{z} \mathcal{P} \right)_{|z=\epsilon \zeta} N^{\mu}$, we have to introduce the following quantity. If $\textbf{A}$ is a vector field on $\Omega_{t}$, we define $\text{A}_{\sslash}$ as

\begin{equation*}
\text{A}_{\sslash} = \frac{1}{\sqrt{\mu}} \underline{\text{A}}_{h} + \epsilon \underline{\text{A}}_{v} \nabla \zeta,
\end{equation*}

\noindent where $\text{A}_{h}$ is the horizontal component of $\text{A}$, $\text{A}_{v}$ its vertical component, $\underline{\text{A}}  = \textbf{A}_{|z=\epsilon \zeta}$  and $\text{A}_{b} = \textbf{A}_{|z=-1+\beta b}$. Notice that,

\begin{equation}\label{Apara_vectorial}
\underline{\text{A}} \times N^{\mu} = \sqrt {\mu} \begin{pmatrix} - \text{A}_{\sslash}^{\perp} \\ -\epsilon \sqrt{\mu} \text{A}_{\sslash}^{\perp} \cdot \nabla  \zeta \end{pmatrix}.
\end{equation}

\noindent Therefore, taking the orthogonal of the horizontal component of the vectorial product of \eqref{eq_int1} with $N^{\mu}$ we obtain

\begin{equation}\label{eq_int2}
\partial_{t} \Umus +  \nabla \zeta + \frac{\epsilon}{2} \nabla \left\lvert \Umus \right\rvert^{2} - \frac{\epsilon}{2 \mu} \nabla \left[ \left( 1 +  \epsilon^{2} \mu \left\lvert \nabla \zeta\right\rvert^{2} \right) \underline{\text{w}}^{2} \right] + \epsilon \left( \underline{\omega} \hspace{-0.05cm} \cdot \hspace{-0.05cm} N^{\mu} \hspace{0.05cm} + \frac{1}{\text{Ro}} \right) \underline{\text{V}}^{\perp} = - \nabla P.
\end{equation}

\noindent Since $\Umus$ is a vector field on $\mathbb{R}^{2}$, we have the classical Hodge-Weyl decomposition

\begin{equation}\label{decomposition_Umus}
\Umus = \nabla \frac{\nabla}{\Delta} \cdot \Umus + \nabla^{\perp} \frac{\nabla^{\perp}}{\Delta} \cdot \Umus.
\end{equation}

\noindent In the following we denote by $\psi := \frac{\nabla}{\Delta} \cdot \Us$ and $\psit := \frac{\nabla^{\perp}}{\Delta} \cdot \Us$ \footnote{We define rigorously these operators in the next section.}. Applying the operator $\frac{\nabla}{\Delta} \cdot$ to \eqref{eq_int2}, we obtain

\begin{equation}
\partial_{t} \psi + \zeta + \frac{\epsilon}{2} \left\lvert \Umus \right\rvert^{2} - \frac{\epsilon}{2 \mu} \left( 1 + \epsilon^{2} \mu \left\lvert \nabla \zeta \right\rvert^{2} \right) \! \underline{\text{w}}^{2} + \epsilon \frac{\nabla}{\Delta} \cdot \left[ \left( \underline{\omega} \cdot N^{\mu} + \frac{1}{\text{Ro}} \right) \underline{\text{V}}^{\perp} \right] = -P.
\end{equation}

\noindent Moreover, using the following vectorial identity

\begin{equation}\label{Uss_curl}
\left( \nabla^{\mu}_{\! X,z} \times \textbf{U}^{\mu} \right)_{|z=\epsilon \zeta} \cdot N^{\mu} = \mu \nabla^{\perp} \cdot \Umus,
\end{equation}

\noindent we have

\begin{equation}
\Delta \psit = \left( \underline{\omega} \cdot N^{\mu} \right).
\end{equation}

\noindent We can now give the nondimensionalized Castro-Lannes formulation of the water waves equations with vorticity in the presence of Coriolis forcing. It is a system of three equations for the unknowns $\left( \zeta,\psi, \bm{\omega} \right)$

\small
\begin{equation}\label{Castro_lannes_formulation}
\left\{
\begin{aligned}
&\hspace{-0.05cm} \partial_{t} \zeta - \frac{1}{\mu} \underline{\text{U}}^{\mu} \cdot N^{\mu} = 0,\\
&\hspace{-0.05cm} \partial_{t} \psi \hspace{-0.05cm} + \hspace{-0.05cm} \zeta \hspace{-0.05cm} + \hspace{-0.05cm} \frac{\epsilon}{2} \hspace{-0.05cm} \left\lvert  \text{U}^{\mu}_{\! \sslash} \right\rvert^{2} \hspace{-0.15cm} - \hspace{-0.05cm} \frac{\epsilon}{2 \mu} \hspace{-0.05cm} \left(\hspace{-0.05cm} 1 + \hspace{-0.05cm} \epsilon^{2} \mu \left\lvert \nabla \zeta \right\rvert^{2} \hspace{-0.05cm} \right) \! \underline{\text{w}}^{2} \hspace{-0.1cm} + \hspace{-0.1cm} \epsilon \frac{\nabla}{\Delta} \hspace{-0.1cm} \cdot \hspace{-0.1cm} \left[ \left( \underline{\omega} \hspace{-0.05cm} \cdot \hspace{-0.05cm} N^{\mu} \hspace{0.05cm} + \frac{1}{\text{Ro}} \right) \underline{\text{V}}^{\perp} \right] = - P,\\
&\hspace{-0.05cm} \partial_{t} \bm{\omega} \hspace{-0.05cm} + \hspace{-0.05cm} \frac{\epsilon}{\mu} \hspace{-0.05cm} \left( \hspace{-0.05cm} \textbf{U}^{\mu} \hspace{-0.05cm} \cdot \hspace{-0.05cm} \nabla^{\mu}_{\! X,z} \hspace{-0.05cm} \right) \bm{\omega} \hspace{-0.05cm} = \hspace{-0.05cm} \frac{\epsilon}{\mu} \left( \bm{\omega} \cdot \nabla^{\mu}_{\! X,z} \right) \textbf{U}^{\mu} \hspace{-0.05cm} + \hspace{-0.05cm} \frac{\epsilon}{\mu \text{Ro}} \partial_{z} \textbf{U}^{\mu},
\end{aligned}
\right.
\end{equation}
\normalsize

\noindent where $\textbf{U}^{\mu} := \textbf{U}^{\mu}[\epsilon \zeta, \beta b](\psi,\bm{\omega})$ is the unique solution in $H^{1}(\Omega_{t})$ of

\begin{equation}\label{div_curl_formulation}
\left\{
\begin{aligned}
&\text{curl}^{\mu} \; \textbf{U}^{\mu} = \mu \bm{\omega} \text{ in } \Omega_{t},\\
&\text{div}^{\mu} \; \textbf{U}^{\mu} = 0 \text{ in } \Omega_{t},\\
&\Us^{\mu} = \nabla \psi + \frac{\nabla^{\perp}}{\Delta} \left( \underline{\omega} \cdot N^{\mu} \right),\\
&\text{U}_{b}^{\mu} \cdot N^{\mu}_{b} = 0.
\end{aligned}
\right.
\end{equation}

\noindent We add a technical assumption. We assume that the water depth is bounded from below by a positive constant

\begin{equation}\label{nonvanishing}
  \exists \, h_{\min} > 0 \text{ ,  } \epsilon \zeta + 1 - \beta b \geq h_{\min}.
\end{equation} 

\noindent We also suppose that the dimensionless parameters satisfy

\begin{equation}\label{constraints_parameters}
\exists \mu_{\max} \text{, } 0 < \mu \leq \mu_{\max} \text{, } 0 < \epsilon \leq 1 \text{, } 0 < \beta \leq 1 \text{ and } \frac{\epsilon}{\text{Ro}} \leq 1.
\end{equation}

\begin{remark}
\noindent The assumption $ \epsilon \leq \text{Ro}$ is equivalent to $fL \leq \sqrt{gH}$. This means that the typical rotation speed due to the Coriolis force is less than the typical water wave celerity. For water waves, this assumption is common (see for instance \cite{pedlosky}). Typically for offshore long water waves at mid-latitudes, we have $L=100 \text{km}$ and $H=1 \text{km}$ and  $f=10^{-4} \text{Hz}$. Then, $\frac{\epsilon}{\text{Ro}} = 10^{-1}$.
\end{remark}

\subsection{Existence result}

\noindent In this part, we give our main result. It is a wellposedness result for the system \eqref{Castro_lannes_formulation_straight} which is a straightened system of the Castro-Lannes formulation.  This result extends Theorem 4.7 and Theorem 5.1 in \cite{Castro_Lannes_vorticity} by adding a non flat bottom and a Coriolis forcing. We define the energy $\mathcal{E}^{N}$ in Subsection \ref{framework_ww}.

\begin{thm}
\noindent Assume that the initial data, $b$ and $P$ are smooth enough and the initial vorticity is divergent free. Assume also that Conditions \eqref{nonvanishing} and \eqref{rayleigh_taylor_assumption} are satisfied initially. Then, there exists \upshape$T > 0$, \itshape and a unique solution to the water waves equations \eqref{Castro_lannes_formulation_straight} on $[0,T]$. Moreover, 

\upshape
\begin{equation*}
T = \min \left( \frac{T_{0}}{\max(\epsilon, \beta, \frac{\epsilon}{\text{Ro}})}, \frac{T_{0}}{ \left\lvert \nabla P \right\rvert_{L^{\infty}_{t} H_{\! X}^{N}}} \right) \text{ , }  \frac{1}{T_{0}} =c^{1} \text{ and  } \underset{t \in [0,T]}{\sup} \mathcal{E}^{N} \hspace{-0.1cm} \left( \zeta(t), \psi(t), \omega(t) \right) = c^{2},
\end{equation*}
\itshape

\noindent where \upshape$c^{j}$ \itshape is a constant which depends on the initial conditions, $P$ and $b$. 
\end{thm}

\noindent A full version is given in Subsection \ref{Existence_result}. This theorem allows us to investigate the justification of asymptotic models in the presence of a Coriolis forcing. In the case of a constant pressure at the surface and without a Coriolis forcing, our existence time is similar to Theorem 3.16 in \cite{Lannes_ww} (see also \cite{Alvarez_Lannes}); without a Coriolis forcing, it is as Theorem 2.3 in \cite{my_proud_res}. 

\subsection{Notations}

\noindent - If $\textbf{A} \in \R^{3}$, we denote by $\textbf{A}_{h}$ its horizontal component and by  $\textbf{A}_{v}$ its vertical component. 

\medskip

\noindent - If $\textbf{V}  = \begin{pmatrix} \text{V}_{1} \\ \text{V}_{2} \end{pmatrix} \in \RR$, we define the orthogonal of $\textbf{V}$ by $\textbf{V}^{\perp} = \begin{pmatrix} -\text{V}_{2} \\ \text{V}_{1} \end{pmatrix}$.

\medskip

\noindent - In this paper, $C \left( \cdot \right)$ is a nondecreasing and positive function whose exact value has non importance.

\medskip

\noindent - Consider a vector field $\textbf{A}$ or a function $\textbf{w}$ defined on $\Omega$. Then, we denote $\text{A} = \textbf{A} \circ \Sigma$ and $\text{w} = \textbf{w} \circ \Sigma$, where $\Sigma$ is defined in \eqref{diffeo}. Furthermore, we denote $\underline{\text{A}} = \textbf{A}_{|z=\epsilon\zeta} = \text{A}_{|z=0}$, $\underline{\text{w}} = \textbf{w}_{|z=\epsilon \zeta} = \text{w}_{|z=0}$ and $\text{A}_{b} = \textbf{A}_{|z=-1+\beta b} = \text{A}_{|z=-1}$, $\text{w}_{b} = \textbf{w}_{|z=-1+b}= \text{w}_{|z=-1}$.

\medskip

\noindent - If $s \in \R$ and $f$ is a function on $\RD$, $\left\lvert f \right\rvert_{H^{s}}$ is its $H^{s}$-norm and $\left\lvert f \right\rvert_{2}$ is its $L^{2}$-norm. The quantity $\left\lvert f \right\rvert_{W^{k,\infty}}$ is $W^{k,\infty}(\RD)$-norm of $f$, where $k \in \mathbb{N}^{\ast}$, and $\lver f \rver_{L^{\infty}}$ its $L^{\infty}(\RD)$-norm.

\medskip

\noindent - The operator $\left( \, , \, \right)$ is the $L^{2}$-scalar product in $\RD$.

\medskip

\noindent - If $N \in \mathbb{N}^{\ast}$, $\textbf{A}$ is defined on $\Omega$ and $\text{A} = \textbf{A} \circ \Sigma$, $\left\lvert\left\lvert \text{A} \right\rvert\right\rvert_{H^{N}}$ and $\left\lvert\left\lvert \textbf{A} \right\rvert\right\rvert_{H^{N}}$ are respectively the $H^{N}(\mathcal{S})$-norm  of $\text{A}$ and the $H^{N}(\Omega)$-norm of $\textbf{A}$. The $L^{p}$-norm are denoted $\llver \cdot \rrver_{p}$.

\medskip

\noindent - The norm $\left\lvert\left\lvert \cdot \right\rvert\right\rvert_{H^{s,k}}$ is defined in Definition \ref{Hsk_spaces}. 

\medskip

\noindent - The space $H^{s}_{\ast}(\RD)$, $\Hdot^{s}(\RD)$ and $H_{b} \hspace{-0.05cm} \left(\text{div}^{\mu}_{0},\Omega \right)$ are defined in Subsection \ref{notations_div_curl}.

\medskip

\noindent - If $f$ is a function defined on $\RD$, we denote $\nabla f$ the gradient of $f$.

\medskip

\noindent - If $\textbf{w}$ is a function defined on $\Omega$, $\nabla_{\! X,z} \textbf{w}$ is the gradient of $\textbf{w}$ and $\nabla_{\! X} \textbf{w}$ its horizontal component. We have the same definition for functions defined on $\mathcal{S}$. 

\medskip

\noindent - $\mathfrak{P}$, $\Lambda$ and $M_{N}$ are defined in Subsection \ref{notations_div_curl}.

\section{The div-curl problem}

\noindent In \cite{Castro_Lannes_vorticity}, A. Castro and D. Lannes study the system \eqref{div_curl_formulation} in the case of a flat bottom ($b=0$). The purpose of this part is to extend their results in the case of a non flat bottom.

\subsection{Notations}\label{notations_div_curl}

\noindent In this paper, we use the Beppo-Levi spaces (see \cite{Deny_Lions_beppo_levi})

\begin{equation*}
\forall s \geq 0 \text{, } \Hdot^{s}(\RD) = \left\{ f \in L^{2}_{\text{loc}}(\RD) \text{, } \nabla f \in H^{s-1}(\RD) \right\} \text{ and } \left\lvert \cdot \right\rvert_{\Hdot^{s}} = \left\lvert \nabla \cdot \right\rvert_{H^{s-1}}.
\end{equation*}

\noindent The dual space of $\Hdot^{s}(\RD)/\raisebox{-.65ex}{\ensuremath{\R}}$ is the space (see \cite{Buffoni_stability})

\begin{equation*}
H^{-s}_{\ast}(\RD) = \left\{ u \in H^{-s}(\RD) \text{, } \exists v \in H^{-s+1}(\RD) \text{, } u = \left\lvert D \right\rvert v \right\} \text{ and } \left\lvert \cdot \right\rvert_{H^{-s}_{\ast}} = \left\lvert \frac{\cdot}{|D|} \right\rvert_{H^{-s+1}}.
\end{equation*}

\noindent Notice that $\Hdot^{1}(\RD)/\raisebox{-.65ex}{\ensuremath{\R}}$ is a Hilbert space. Then, we can rigorously define the Hodge-Weyl decomposition and the operators $\frac{\nabla}{\Delta} \cdot$ and $\frac{\nabla^{\perp}}{\Delta} \cdot$. For $f \in L^{2}(\RD)^{d}$, $u = \frac{\nabla}{\Delta} \cdot f$ is defined as the unique solution, up to a constant, in $\Hdot^{1}(\RD)$ of the variational problem

\begin{equation*}
\int_{\RD} \nabla u \cdot \nabla \phi = \int_{\RD} f \cdot \nabla \phi \text{ , } \forall \phi \in  \Hdot^{1}(\RD). 
\end{equation*}

\noindent The operator $\frac{\nabla^{\perp}}{\Delta} \cdot$ can be defined similarly. Then, it is easy to check that the operators $\frac{\nabla^{\perp}}{\Delta} \cdot$ and $\frac{\nabla^{\perp}}{\Delta} \cdot$  belong to  $\mathcal{L} \left(H^{s}(\RD)^{d}, \Hdot^{s+1}(\RD) \right)$, for all $s \geq 0$.

\noindent The subspace of $L^{2}(\Omega)^{3}$ of functions whose rotationnal is in $L^{2}(\Omega)^{3}$ is the space

\begin{equation*}
H \hspace{-0.05cm} \left(\text{curl}^{\mu},\Omega \right) = \left\{ \textbf{A} \in L^{2}(\Omega)^{3} \text{, } \text{curl}^{\mu} \textbf{A} \in L^{2}(\Omega)^{3} \right\}.
\end{equation*}

\noindent The subspace of $L^{2}(\Omega)^{3}$ of divergence free vector fields is the space

\begin{equation*}
H \hspace{-0.05cm} \left(\text{div}^{\mu}_{0},\Omega \right) = \left\{ \textbf{A} \in L^{2}(\Omega)^{3} \text{, } \text{div}^{\mu} \textbf{A} = 0 \right\}.
\end{equation*}

\begin{remark}
\noindent Notice that \upshape $\textbf{A} \in H \hspace{-0.05cm} \left(\text{div}^{\mu}_{0},\Omega \right)$ \itshape implies that \upshape $\left( \textbf{A}_{|\partial \Omega} \cdot n \right)$ \itshape belongs to \upshape $H^{-\frac{1}{2}}(\partial \Omega)$ \itshape and \upshape $\textbf{A} \in H \hspace{-0.05cm} \left(\text{curl}^{\mu},\Omega \right)$ \itshape implies that \upshape $\left( \textbf{A}_{|\partial \Omega} \times n \right)$ \itshape belongs to \upshape $H^{-\frac{1}{2}}(\partial \Omega)$ \itshape (see \cite{Dautray_Lions}).
\end{remark}

\medskip

\noindent Finally, we define $H_{b} \hspace{-0.05cm} \left(\text{div}^{\mu}_{0},\Omega \right)$  as

\medskip

\begin{equation*}
H_{b} \hspace{-0.05cm} \left(\text{div}^{\mu}_{0},\Omega \right) = \left\{ \textbf{A} \in H \hspace{-0.05cm} \left(\text{div}^{\mu}_{0},\Omega \right) \text{, } A_{b} \cdot N^{\mu}_{b} \in H^{-\frac{1}{2}}_{\ast}(\RD) \right\}.
\end{equation*}

\begin{remark}
\noindent We have a similar equation to \eqref{Uss_curl} at the bottom

\upshape
\begin{equation*}
\frac{1}{\mu} \left( \nabla^{\mu}_{\! X,z} \times \textbf{U}^{\mu} \right)_{|z=-1 + \beta b} \cdot N^{\mu}_{b} = \nabla^{\perp} \cdot \left(\textbf{V}_{b} + \beta \textbf{w}_{b} \nabla b \right),
\end{equation*}
\itshape

\noindent hence, in the following, we suppose that \upshape $\bm{\omega} \in H_{b} \hspace{-0.05cm} \left(\text{div}^{\mu}_{0},\Omega \right)$ \itshape.
\end{remark}

\noindent We define $\mathfrak{P}$ and $\Lambda$ as the Fourier multiplier in $\mathcal{S}^{\prime} \hspace{-0.05cm} \left(\RD \right)$,

\begin{equation*}
\mathfrak{P} = \frac{\left\lvert D \right\rvert}{\sqrt{1+\sqrt{\mu} \left\lvert D \right\rvert}} \text{ and } \Lambda = \sqrt{1+\left\lvert D \right\rvert^{2}}.
\end{equation*}

\noindent Then it is important to notice that, if $\bm{\omega} \in H_{b} \hspace{-0.05cm} \left(\text{div}^{\mu}_{0},\Omega \right)$, the quantity $\frac{1}{\mathfrak{P}} \hspace{-0.08cm} \left(\omega_{b} \cdot N^{\mu}_{b} \right)$ makes sense and belongs to $L^{2} \! \left(\RD\right)$.
\medskip

\noindent In the following $M_{N}$ is a constant of the form

\begin{equation}\label{M_N_def}
M_{N} = C \left( \mu_{\max}, \frac{1}{h_{\min}}, \epsilon \left\lvert \zeta \right\rvert_{H^{N}}, \beta \left\lvert \nabla b \right\rvert_{H^{N}}, \beta \left\lvert b \right\rvert_{L^{\infty}} \right).
\end{equation}

\subsection{Existence and uniqueness}

\noindent In this part, we forget the dependence on $t$. First, notice that we can split the problem into two part. Let $\Phi \in \Hdot^{2}(\Omega)$ the unique solution of the Laplace problem (see \cite{Lannes_ww})

\begin{equation}\label{laplace_problem}
\left\{
\begin{aligned}
& \Delta^{\mu}_{\! X,z} \Phi = 0 \text{ in } \Omega,\\
&\Phi_{|z=\epsilon \zeta} = \psi \text{, } \left( N_{b}^{\mu} \hspace{-0.05cm} \cdot \hspace{-0.05cm} \nabla^{\mu}_{\! X,z} \Phi \right)_{|z=-1+\beta b} = 0.
\end{aligned}
\right.
\end{equation}

\noindent Using the vectorial identity

\begin{equation*}
\left(\nabla_{\! X,z}^{\mu} \Phi \right)_{\! \sslash} = \nabla \psi,
\end{equation*}

\noindent it is easy to check that if $\textbf{U}^{\mu}$ satisfies \eqref{div_curl_formulation},  $\widetilde{\textbf{U}}^{\mu} := \textbf{U}^{\mu} - \nabla_{\! X,z}^{\mu} \Phi$ satisfies

\begin{equation}\label{div_curl_formulation2}
\left\{
\begin{aligned}
&\text{curl}^{\mu} \; \widetilde{\textbf{U}}^{\mu} = \mu \: \bm{\omega} \text{ in } \Omega_{t},\\
&\text{div}^{\mu} \; \widetilde{\textbf{U}}^{\mu} = 0 \text{ in } \Omega_{t},\\
&\widetilde{\text{U}}_{\! \sslash}^{\mu} = \frac{\nabla^{\perp}}{\Delta} \left( \underline{\bm{\omega}} \cdot N^{\mu} \right) \text{at the surface},\\
&\widetilde{\textbf{U}}^{\mu}_{b} \cdot N^{\mu}_{b} = 0 \text{ at the bottom}.
\end{aligned}
\right.
\end{equation}

\noindent In the following we focus on the system \eqref{div_curl_formulation2}. We give 4 intermediate results in order to get the existence and uniqueness. The first Proposition shows how to control the norm of the gradient of a function with boundary condition as in \eqref{div_curl_formulation2}.

\begin{prop}\label{formulation_nabla}
\noindent Let \upshape $\zeta \text{, } b \in W^{2,\infty} \hspace{-0.05cm} \left( \RD \right)$, $\textbf{A} \in H(\text{div}^{\mu}_{0},\Omega) \cap H(\text{curl}^{\mu},\Omega)$. \itshape Then,  for all \upshape $C \in H^{1}(\RD)^{3}$, \itshape we have

\upshape
\begin{equation}
\int_{\Omega} \hspace{-0.2cm} \nabla_{\! X,z}^{\mu} \textbf{A} \hspace{-0.05cm}:\hspace{-0.05cm} \nabla_{\! X,z}^{\mu} \textbf{C} = \hspace{-0.2cm} \int_{\Omega} \hspace{-0.2cm} \text{curl}^{\; \mu} \textbf{A} \hspace{-0.05cm} : \hspace{-0.05cm} \text{curl}^{\; \mu} \textbf{C} + \left< l^{\mu}[ \epsilon \zeta] \! \left( \underline{\text{A}} \right) \hspace{-0.05cm} \text{, } \hspace{-0.05cm} \underline{C} \right>_{H^{-\frac{1}{2}} - H^{\frac{1}{2}}}  - \left<l^{\mu}[ \beta b]\! \left( \text{A}_{b} \right) \hspace{-0.05cm} \text{, } \hspace{-0.05cm} C_{b} \right>_{H^{-\frac{1}{2}} - H^{\frac{1}{2}}},
\end{equation}
\itshape

\noindent where for \upshape $\text{B} \in H^{\frac{1}{2}}\left(\RR \right)^{3}$ \itshape and for \upshape $\eta \in W^{2,\infty}(\RD)$,  

\begin{equation}
l^{\mu}[\eta] \! \left(\text{B} \right) = \begin{pmatrix} \sqrt{\mu} \nabla \text{B}_{v} - \mu \left( \nabla^{\perp} \eta \cdot \nabla \right) \text{B}_{h}^{\perp} \\ - \sqrt{\mu} \nabla \cdot \text{B}_{h} \end{pmatrix}.
\end{equation}
\itshape

\noindent Furthermore, if \upshape $\psit \in \Hdot^{\frac{3}{2}}(\RD)$ \itshape and

\upshape
\begin{equation*}
A_{b} \cdot N^{\mu}_{b} = 0 \text{ and } \text{A}_{\sslash} = \nabla^{\perp} \psit,
\end{equation*}
\itshape

\noindent we have the following estimate

\upshape

\begin{equation}\label{grad_norm_A_estimate}
\begin{aligned}
\llver \nabla^{\mu}_{\! X,z} \textbf{A} \rrver^{2}_{2} \hspace{-0.1cm} \leq &\hspace{-0.05cm} \llver \text{curl}^{\; \mu} \textbf{A} \rrver^{2}_{2} + \mu C \left(\epsilon \left\lvert \zeta  \right\rvert_{W^{2,\infty}}\hspace{-0.1cm}, \beta \left\lvert b \right\rvert_{W^{2,\infty}} \right) \hspace{-0.05cm} \left(\left\lvert \underline{\text{A}} \right\rvert_{2}^{2} +  \left\lvert \text{A}_{bh} \right\rvert_{2}^{2} \right)\\
&+ \mu C \left(\mu_{\max}, \epsilon \left\lvert \zeta  \right\rvert_{W^{2,\infty}}\hspace{-0.1cm}, \beta \left\lvert b \right\rvert_{W^{2,\infty}} \right)  \left\lvert \sqrt{1+\sqrt{\mu} |D|} \nabla \psit \right\rvert_{2} \hspace{-0.05cm} \left\lvert \sqrt{1+\sqrt{\mu} \left\lvert D \right\rvert } \underline{\text{A}}_{h} \right\rvert_{2}.
\end{aligned}
\end{equation}
\itshape

\end{prop}

\begin{proof}

\noindent Using the Einstein summation convention and denoting $\nabla^{\mu}_{\! X,z} = \left( \partial^{\mu}_{1}, \partial^{\mu}_{2}, \partial^{\mu}_{3} \right)^{T}$, a simple computation gives (see Lemma 3.2 in \cite{Castro_Lannes_vorticity} or Chapter 9 in \cite{Dautray_Lions}),

\begin{equation}\label{estimate_div_curl}
\llver \nabla^{\mu} \textbf{A} \rrver^{2}_{2} = \llver \text{curl}^{\mu} \textbf{A} \rrver_{2}^{2} + \llver \text{div}^{\mu} \textbf{A} \rrver_{2}^{2} + \int_{\partial \Omega} n^{\mu}_{i} \textbf{A}_{j} \partial^{\mu}_{j} \textbf{A}_{i} - n^{\mu}_{j} \textbf{A}_{j} \partial^{\mu}_{i} \textbf{A}_{i}.
\end{equation}

\noindent In this case, $\partial \Omega$ is the union of two surfaces and $\overset{\rightarrow}{n^{\mu}} = \pm \begin{pmatrix} - \sqrt{\mu} \nabla \eta \\ 1 \end{pmatrix}$, where $\eta$ is the corresponding surface. Then, one can check that (see also Lemma 3.8 in \cite{Castro_Lannes_vorticity}),

\begin{equation}\label{boundary_expression}
\int_{\{z=\eta \}} \hspace{-0.7cm} n^{\mu}_{i} \textbf{A}_{j} \partial^{\mu}_{j} \textbf{A}_{i} - n^{\mu}_{j} \textbf{A}_{j} \partial^{\mu}_{i} \textbf{A}_{i} \hspace{-0.05cm} = \hspace{-0.05cm} \pm \hspace{-0.1cm} \int_{\RD} \textbf{A}_{\eta,h} \cdot \left( 2 \sqrt{\mu} \nabla_{\! X} \text{A}_{\eta,v} - \mu \left( \nabla \eta^{\perp} \cdot \nabla \right) \textbf{A}_{\eta,h}^{\perp} \right),
\end{equation}

\noindent where $\textbf{A}_{\eta} := \textbf{A}_{|z=\eta}$. The first part of the Proposition follows by polarization of Equations \eqref{estimate_div_curl} and \eqref{boundary_expression} (as quadratic forms). For the second estimate, since $\text{A}_{b} \cdot N^{\mu}_{b} = 0$, we get at the bottom that

\small
\begin{equation*}
\begin{aligned}
\int_{\{z=-1+\beta b \}} \hspace{-1.5cm} n^{\mu}_{i} \textbf{A}_{j} \partial^{\mu}_{j} \textbf{A}_{i} \hspace{-0.06cm}-\hspace{-0.06cm} n^{\mu}_{j} \textbf{A}_{j} \partial^{\mu}_{i} \textbf{A}_{i} \hspace{-0.15cm} &= \hspace{-0.1cm} - 2 \hspace{-0.1cm} \int_{\RD} \hspace{-0.3cm} \mu \beta \hspace{-0.05cm} \left( \hspace{-0.05cm} \partial_{x} b \text{A}_{bx} \partial_{y} \text{A}_{by} \hspace{-0.06cm}+\hspace{-0.06cm} \partial_{y} b \text{A}_{by} \partial_{x} \text{A}_{bx} \hspace{-0.06cm}+\hspace{-0.06cm} \partial^{2}_{xy} b \text{A}_{bx} \text{A}_{by}\right) \hspace{-0.06cm} - \hspace{-0.06cm} \text{A}_{z} \sqrt{\mu} \text{div}_{X} \text{A}_{bh}\\
&=- \mu \beta \int_{\RD} \partial^{2}_{x} b \: \text{A}_{bx}^{2} + \partial^{2}_{y} b \: \text{A}_{by}^{2} + 2 \partial^{2}_{xy} b \: \text{A}_{bx} \text{A}_{by}.
\end{aligned}
\end{equation*}
\normalsize

\noindent At the surface, since $\underline{\text{A}}_{h} = \sqrt{\mu} \nabla^{\perp} \psit - \epsilon \sqrt{\mu} \underline{\text{A}}_{v} \nabla \zeta$, we have

\small
\begin{equation*}
\begin{aligned}
\int_{\{z= \epsilon \zeta \}} \hspace{-0.8cm} n^{\mu}_{i} \textbf{A}_{j} \partial^{\mu}_{j} \textbf{A}_{i} - n^{\mu}_{j} \textbf{A}_{j} \partial^{\mu}_{i} \textbf{A}_{i} \hspace{-0.1cm} &= \hspace{-0.1cm} - 2 \hspace{-0.1cm} \int_{\RD} \hspace{-0.3cm} \epsilon \mu \hspace{-0.05cm} \left( \partial_{x} \zeta \underline{\text{A}}_{y} \partial_{y} \underline{\text{A}}_{x} \hspace{-0.08cm}+\hspace{-0.08cm} \partial_{y} \zeta \underline{\text{A}}_{x} \partial_{x} \underline{\text{A}}_{y} \hspace{-0.08cm}+\hspace{-0.08cm} \partial_{xy}^{2} \zeta \underline{\text{A}}_{x} \underline{\text{A}}_{y} \right) \hspace{-0.08cm}+\hspace{-0.08cm} \sqrt{\mu} \left(\underline{\text{A}}_{h} \hspace{-0.1cm} \cdot \hspace{-0.1cm} \nabla_{\! X} \right) \underline{\text{A}}_{z}\\
&=\epsilon \mu \hspace{-0.1cm} \int_{\RD} \hspace{-0.2cm} \underline{\text{A}}_{x}^{2} \partial_{y}^{2} \zeta + \underline{\text{A}}_{y}^{2} \partial^{2}_{x} \zeta -2 \underline{\text{A}}_{x} \underline{\text{A}}_{y} \partial_{xy}^{2} \zeta + \underline{\text{A}}_{z}^{2} \left[ \partial_{x}^{2} \zeta + \partial_{y}^{2} \zeta \right]\\
&\quad - 2 \epsilon \mu^{\frac{3}{2}} \hspace{-0.1cm} \int_{\RD} \hspace{-0.2cm} \underline{\text{A}}_{h} \cdot \nabla^{\perp} \hspace{-0.1cm} \left( \nabla \psit \cdot \nabla \zeta \right).
\end{aligned}
\end{equation*}
\normalsize

\noindent Then,

\begin{equation*}
\begin{aligned}
\left\lvert 2 \epsilon \mu^{\frac{3}{2}} \hspace{-0.1cm} \int_{\RD} \hspace{-0.2cm} \underline{\text{A}}_{h} \cdot \nabla^{\perp} \hspace{-0.1cm} \left( \nabla \psit \cdot \nabla \zeta \right) \right\rvert \leq \epsilon \mu \lver \sqrt{1+\sqrt{\mu} |D|} \; \underline{\text{A}}_{h} \rver_{2} \lver \sqrt{\mu} \mathfrak{P} \left( \nabla \psit \cdot \nabla \zeta \right) \rver_{2}.
\end{aligned}
\end{equation*}

\noindent and estimate \eqref{grad_norm_A_estimate} follows easily from Lemma \ref{P_product}. 

\end{proof}

\noindent The second Proposition gives a control of the $L^{2}$-norm of the trace.

\begin{prop}\label{L2_trace}
\noindent Let \upshape $\zeta \text{, } b \in W^{1,\infty} \hspace{-0.05cm} \left( \RD \right)$, $\textbf{A} \in H(\text{div}^{\mu}_{0},\Omega) \cap H(\text{curl}^{\mu},\Omega)$ \itshape and \upshape $\psit \in \hspace{-0.05cm} \Hdot^{1}\left( \RD \right)$ \itshape such that

\upshape
\begin{equation*}
\text{A}_{b} \cdot N^{\mu}_{b} = 0 \text{ and } \text{A}_{\sslash} = \nabla^{\perp} \psit.
\end{equation*}
\itshape

\noindent Then, 

\upshape
\begin{equation}
\left\lvert \underline{\text{A}} \right\rvert^{2}_{2} + \left\lvert \text{A}_{b} \right\rvert^{2}_{2} \leq C . \left( \mu \left\lvert \nabla \psit \right\rvert_{2}^{2} + \llver \text{curl}^{\; \mu} \textbf{A} \rrver_{2} \llver \textbf{A} \rrver_{2} \right).
\end{equation}
\itshape

\end{prop}

\begin{proof}
\noindent Using the fact that $\partial_{z} \textbf{A}_{h} = - \left( \text{curl}^{\; \mu} \textbf{A}\right)_{h}^{\perp} + \sqrt{\mu} \nabla_{\! X} \textbf{A}_{v}$, we have

\begin{align*}
\int_{\RD} \hspace{-0.2cm} \left\lvert \underline{\text{A}}_{h} \right\rvert^{2} &= \hspace{-0.05cm} \int_{\RD} \left\lvert \text{A}_{bh} \right\rvert^{2} +2 \int_{\Omega} \partial_{z} \textbf{A}_{h} \cdot \textbf{A}_{h}\\
&= \hspace{-0.05cm} \int_{\RD} \left\lvert \text{A}_{bh} \right\rvert^{2} - 2 \int_{\Omega} \left( \text{curl}^{\; \mu} \textbf{A}\right)_{h}^{\perp} \cdot \textbf{A}_{h} + 2 \sqrt{\mu} \int_{\Omega} \nabla_{\! X} \textbf{A}_{v} \cdot \textbf{A}_{h}\\
&= \hspace{-0.05cm} \int_{\RD} \hspace{-0.2cm} \left\lvert \text{A}_{bh} \right\rvert^{2} \hspace{-0.05cm}-\hspace{-0.05cm} 2 \hspace{-0.05cm} \int_{\Omega} \hspace{-0.2cm} \left( \text{curl}^{\; \mu} \textbf{A}\right)_{h}^{\perp} \cdot \textbf{A}_{h} \hspace{-0.05cm}+\hspace{-0.05cm} 2 \hspace{-0.05cm} \int_{\Omega} \hspace{-0.2cm} \partial_{z} \textbf{A}_{v} \textbf{A}_{v} \hspace{-0.05cm} + 2 \sqrt{\mu} \Bigg( \hspace{-0.05cm} \int_{\RD} \hspace{-0.25cm} \beta \left( \nabla b \hspace{-0.05cm} \cdot \hspace{-0.05cm} \text{A}_{bh} \right) \text{A}_{bv} \hspace{-0.05cm}\\
&\hspace{9.3cm} - \int_{\RD} \hspace{-0.25cm} \left( \epsilon \nabla \zeta \hspace{-0.05cm} \cdot \hspace{-0.05cm} \underline{\text{A}}_{h} \right) \underline{\text{A}}_{v} \hspace{-0.05cm} \hspace{-0.05cm} \Bigg),
\end{align*}

\noindent where the third equality is obtained by integrating by parts the third integral and by using the fact that $\text{div}^{\mu} \textbf{A} = 0$.  Furthermore, thanks to the boundary conditions and Equality \eqref{Apara_vectorial}, we have

\begin{equation*}
\epsilon \sqrt{\mu} \left( \nabla \zeta \hspace{-0.05cm} \cdot \hspace{-0.05cm} \underline{\text{A}}_{h} \right) \underline{\text{A}}_{v} = \sqrt{\mu} \nabla^{\perp} \psit \cdot \underline{\text{A}}_{h} - \left\lvert \underline{\text{A}}_{h} \right\rvert^{2} \text{  and  } \beta \sqrt{\mu} \left( \nabla b \cdot \text{A}_{bh} \right) \text{A}_{bv} = \text{A}_{bv}^{2}.
\end{equation*}

\noindent Then, we get

\begin{equation}
\left\lvert \underline{\text{A}} \right\rvert_{2}^{2} + \left\lvert \text{A}_{b} \right\rvert_{2}^{2} = 2 \sqrt{\mu} \int_{\RD} \hspace{-0.2cm} \nabla^{\perp} \psit \cdot \underline{\text{A}}_{h} + 2 \int_{\RD} \hspace{-0.2cm} \left( \text{curl}^{\; \mu} \textbf{A}\right)_{h}^{\perp} \cdot \textbf{A}_{h},
\end{equation}

\noindent and the inequality follows.
\end{proof}

\noindent The third Proposition is a Poincar\'e inequality.

\begin{prop}\label{Poincare}
\noindent Let \upshape $\zeta \text{, } b \in W^{1,\infty} \hspace{-0.05cm} \left( \RD \right)$ \itshape and \upshape $\textbf{A} \in H(\text{div}^{\mu}_{0},\Omega) \cap H(\text{curl}^{\mu},\Omega)$ \itshape such that

\upshape
\begin{equation*}
\text{A}_{b} \times N^{\mu}_{b} = 0 \text{ and } \underline{\text{A}} \cdot N^{\mu} = 0.
\end{equation*}
\itshape

\noindent Then, 

\upshape
\begin{equation}
\llver \textbf{A} \rrver_{2} \leq C \left\lvert \epsilon \zeta - \beta b + 1 \right\rvert_{L^{\infty}} \left( \llver \text{curl}^{\mu} \textbf{A} \rrver_{2} + \llver \partial_{z} \textbf{A} \rrver_{2} \right).
\end{equation}
\itshape

\end{prop}

\begin{proof}
\noindent We have

\begin{equation*}
\left\lvert \textbf{A}(X,z) \right\rvert^{2} = \left\lvert \text{A}_{b}(X) \right\rvert^{2} + 2 \int_{s=-1+\beta b(X)}^{z} \partial_{z} \textbf{A}(X,s) \cdot \textbf{A}(X,s) dX ds.
\end{equation*}

\noindent Then, the result follows from the following lemma, which is a similar computation to the one in Proposition \ref{L2_trace}.

\begin{lemma}
\noindent Let \upshape $\zeta \text{, } b \in W^{1,\infty} \hspace{-0.05cm} \left( \RD \right)$, $\textbf{A} \in H(\text{div}^{\mu}_{0},\Omega) \cap H(\text{curl}^{\mu},\Omega)$ \itshape such that

\upshape
\begin{equation*}
\text{A}_{b} \times N^{\mu}_{b} = 0 \text{ and } \underline{\text{A}} \cdot N^{\mu} = 0.
\end{equation*}
\itshape
\noindent Then, 

\upshape
\begin{equation}
\left\lvert \underline{\text{A}} \right\rvert^{2}_{2} + \left\lvert \text{A}_{b} \right\rvert^{2}_{2} \leq C \llver \text{curl}^{\; \mu} \textbf{A} \rrver_{2} \llver \textbf{A} \rrver_{2}.
\end{equation}
\itshape
\end{lemma}

\end{proof}

\noindent Finally, the fourth Proposition links the regularity of $\psit$ to the regularity of $\omega_{b} \cdot N^{\mu}_{b}$.

\begin{prop}\label{Control_psit}
\noindent Let \upshape $\zeta, b \in W^{1,\infty}(\RD)$ \itshape be such that Condition \eqref{nonvanishing} is satisfied and let \upshape $\bm{\omega} \in \Hb$\itshape. Then, there exists a unique solution \upshape $\psit \in \Hdot^{\frac{3}{2}}(\RD)$ \itshape to the equation \upshape $\Delta \psit = \underline{\bm{\omega}} \cdot N^{\mu}$ \itshape and we have

\upshape
\begin{equation*}
\left\lvert \nabla \psit \right\rvert_{2} \leq \sqrt{\mu} C \hspace{-0.05cm} \left(  \frac{1}{\text{h}_{\min}} \text{, } \epsilon \left\lvert \zeta \right\rvert_{W^{1,\infty}} \text{, } \beta \left\lvert b \right\rvert_{W^{1,\infty}} \right) \left( \llver \bm{\omega} \rrver_{2} + \frac{1}{\sqrt{\mu}} \lver \frac{1}{\mathfrak{P}} \left( \omega_{b} \cdot N^{\mu}_{b} \right) \rver_{2} \right),
\end{equation*} 
\itshape

\noindent and

\upshape
\small
\begin{equation*}
\left\lvert \sqrt{1+\sqrt{\mu} |D|} \nabla \psit \right\rvert_{2} \leq \sqrt{\mu} C \hspace{-0.05cm} \left(  \frac{1}{\text{h}_{\min}} \text{, } \epsilon \left\lvert \zeta \right\rvert_{W^{1,\infty}} \text{, } \beta \left\lvert b \right\rvert_{W^{1,\infty}} \right) \left( \llver \bm{\omega} \rrver_{2} + \frac{1}{\sqrt{\mu}} \left\lvert \frac{1}{\mathfrak{P}} \left( \omega_{b} \cdot N^{\mu}_{b} \right) \right\rvert_{2} \right).
\end{equation*} 
\normalsize
\itshape
\end{prop}

\begin{proof}
\noindent The proof is a small adaptation of Lemma 3.7 and Lemma 5.5 in \cite{Castro_Lannes_vorticity}.
\end{proof}

\noindent We can now prove an existence and uniqueness result for the system \eqref{div_curl_formulation} and \eqref{div_curl_formulation2}.

\begin{thm}\label{control_U}
\noindent Let \upshape $\zeta \text{, } b \in W^{2,\infty} \hspace{-0.05cm} \left( \RD \right)$ \itshape such that Condition \eqref{nonvanishing} is satisfied, \upshape $\psi \in \Hdot^{\frac{3}{2}} \left( \RD \right)$ \itshape and \upshape $\bm{\omega} \in \Hb$ \itshape. There exists a unique solution \upshape $\textbf{U}^{\mu} = \textbf{U}^{\mu}[\epsilon \zeta,\beta b](\psi,\bm{\omega}) \in H^{1} \left( \Omega \right)$ \itshape to \eqref{div_curl_formulation}. Furthermore, \upshape $\textbf{U}^{\mu} = \nabla_{\! X,z}^{\mu} \Phi + \text{curl}^{\mu} \textbf{A}$ \itshape, where \upshape $\Phi$ \itshape satisfies \eqref{laplace_problem} and \upshape $\textbf{A}$ \itshape satisfies

\upshape
\begin{equation}\label{div_curl_formulation3}
\left\{
\begin{aligned}
&\text{curl}^{\mu} \text{curl}^{\mu} \textbf{A} = \mu \: \bm{\omega} \text{ in } \Omega_{t},\\
&\text{div}^{\mu} \textbf{A} = 0 \text{ in } \Omega_{t},\\
&N^{\mu}_{b} \times \text{A}_{b} = 0,\\
&N^{\mu} \cdot \underline{\text{A}} = 0,\\
&\left( \text{curl}^{\mu} \textbf{A} \right)_{\sslash} = \frac{\nabla^{\perp}}{\Delta} \left( \underline{\omega} \cdot N^{\mu} \right),\\
&N^{\mu}_{b} \cdot \left( \text{curl}^{\; \mu} \textbf{A} \right)_{|z=-1+\beta b} = 0.
\end{aligned}
\right.
\end{equation}
\itshape

\noindent Finally, one has

\upshape
\small
\begin{equation}
\llver \textbf{U}^{\mu} \rrver_{2} \leq \sqrt{\mu} C \left(\mu_{\max}, \frac{1}{\text{h}_{\min}}, \epsilon \left\vert \zeta \right\rvert_{W^{2,\infty}}, \beta \left\vert b \right\rvert_{W^{2,\infty}} \right) \left( \sqrt{\mu} \llver \bm{\omega} \rrver_{2} + \left\lvert \frac{1}{\mathfrak{P}} \left( \omega_{b} \cdot N^{\mu}_{b} \right) \right\rvert_{2} + \left\vert \mathfrak{P} \psi \right\rvert_{2} \right),
\end{equation}
\normalsize
\itshape

\noindent and

\upshape
\small
\begin{equation}
\llver \nabla^{\mu}_{\! X,z} \textbf{U}^{\mu} \rrver_{2} \leq \mu C \left(\mu_{\max},  \frac{1}{\text{h}_{\min}}, \epsilon \left\vert \zeta \right\rvert_{W^{2,\infty}}, \beta \left\vert b \right\rvert_{W^{2,\infty}} \right) \left( \llver \bm{\omega} \rrver_{2} + \left\lvert \frac{1}{\mathfrak{P}} \left( \omega_{b} \cdot N^{\mu}_{b} \right) \right\rvert_{2} + \left\vert \mathfrak{P} \psi \right\rvert_{H^{1}} \right).
\end{equation}
\normalsize
\itshape
\end{thm}

\begin{proof}
\noindent The uniqueness follows easily from the last Propositions. The existence of $\Phi$ and the control of its norm are proved in Chapter 2 in \cite{Lannes_ww}. We focus on the existence of a solution of \eqref{div_curl_formulation3}. The main idea is the following variational formulation for the system \eqref{div_curl_formulation3} (we refer to Lemma 3.5 and Proposition 5.3 in \cite{Castro_Lannes_vorticity} for the details). We denote by 

\begin{equation*}
\mathcal{X} = \left\{ \textbf{C} \in H^{1} \left( \Omega \right) \text{, } \text{div}^{\, \mu} \textbf{C} = 0 \text{, } \underline{\text{A}} \cdot N^{\mu} = 0 \text{ and } \text{A}_{b} \times N^{\mu}_{b} = 0  \right\},
\end{equation*}

\noindent and $\psit$ the unique solution in $\Hdot^{1} (\RD)$ of $\Delta \psit = \underline{\bm{\omega}} \cdot N^{\mu}$. Then, $\text{A} \in \mathcal{X}$ is a variational solution of System \eqref{div_curl_formulation3} if

\begin{equation}\label{A_variational_formulation}
\forall \textbf{C} \in \mathcal{X} \text{, } \int_{\Omega} \text{curl}^{\mu} \; \textbf{A} \cdot \text{curl}^{\mu} \; \textbf{C} = \mu \hspace{-0.1cm} \int_{\Omega} \bm{\omega} \cdot \textbf{C} + \mu \hspace{-0.1cm} \int_{\RD} \nabla \psit \cdot \text{C}_{\sslash},
\end{equation}

\noindent The existence of such a $\textbf{A}$ follows Lax-Milgram's theorem. In the following we only explain how we get the coercivity. Thanks to a similar computation that we used to prove Estimate \eqref{grad_norm_A_estimate} we get

\begin{equation*}
\llver \nabla^{\mu}_{\! X,z} \textbf{A} \rrver^{2}_{2} \hspace{-0.1cm} \leq \llver \text{curl}^{\; \mu} \textbf{A} \rrver^{2}_{2} + \mu C \left(\epsilon \left\lvert \nabla \zeta  \right\rvert_{W^{2,\infty}}\hspace{-0.1cm}, \beta \left\lvert  \nabla b \right\rvert_{W^{2,\infty}} \right) \hspace{-0.05cm} \left(\left\lvert \underline{\text{A}} \right\rvert_{2}^{2} +  \left\lvert \text{A}_{bh} \right\rvert_{2}^{2} \right).
\end{equation*}

\noindent Then, thanks to the similar computation that in Proposition \ref{L2_trace} and Proposition \ref{Poincare} we obtain the coercivity

\begin{equation*}
\llver \textbf{A} \rrver_{2} + \llver \nabla^{\mu}_{\! X,z} \textbf{A} \rrver_{2} \leq C \left(\mu_{\max},  \frac{1}{h_{\min}}, \epsilon \left\lvert \zeta \right\rvert_{W^{2,\infty}}, \beta \left\lvert b \right\rvert_{W^{2,\infty}} \right) \llver \text{curl}^{\mu} \textbf{A} \rrver_{2}.
\end{equation*} 

\noindent Then, we can easily extend this for all $\textbf{C}$ in $\left\{ \textbf{C} \in H^{1} \left( \Omega \right) \text{, } \underline{\text{C}} \cdot N^{\mu} = 0 \text{ and } \text{C}_{b} \times N^{\mu}_{b} = 0  \right\}$ and thanks to the variational formulation of $\textbf{A}$ we get

\begin{equation*}
\llver \text{curl}^{\mu} \textbf{A} \rrver_{2} \leq C \left(\mu_{\max},  \frac{1}{h_{\min}}, \epsilon \left\lvert \zeta \right\rvert_{W^{2,\infty}}, \beta \left\lvert b \right\rvert_{W^{2,\infty}} \right) \left( \mu \llver \bm{\omega} \rrver_{2} + \sqrt{\mu} \left\lvert \nabla \psit \right\rvert_{2} \right). 
\end{equation*} 

\noindent Using Proposition \ref{Control_psit}, we get the first estimate. The second estimate follows from the first estimate, the inequality \eqref{grad_norm_A_estimate}, Proposition 2.4, Proposition 2.6 and the following Lemma.

\begin{lemma}\label{boundary_control}
\noindent Let \upshape $\zeta \text{, } b \in W^{1,\infty} \hspace{-0.05cm} \left( \RD \right)$ \itshape be such that Condition \eqref{nonvanishing} is satisfied. Then, for all \upshape $u \in H^{1}\left( \Omega \right)$,

\small
\begin{equation*}
\left\lvert \sqrt{1+\sqrt{\mu} |D|} \underline{u} \right\rvert_{2} + \left\lvert \sqrt{1+\sqrt{\mu} |D|} u_{b} \right\rvert_{2} \leq C \left( \frac{1}{h_{\min}}, \epsilon \left\lvert \zeta \right\rvert_{W^{1,\infty}}, \beta \left\lvert b \right\rvert_{W^{1,\infty}} \right) \left( \llver \nabla^{\mu}_{\! X,z} u \rrver_{2} + \llver u \rrver_{2} \right).
\end{equation*}
\normalsize
\itshape

\end{lemma}

\begin{proof}
\noindent The proof is a small adaptation of Lemma 5.4 in \cite{Castro_Lannes_vorticity}.
\end{proof}
\end{proof}

\subsection{The transformed div-curl problem}\label{transformed_div_curl}\label{transformed_divcurl}

\noindent In this section, we transform the div-curl problem in the domain $\Omega$ into a variable coefficients problem in the flat strip $\mathcal{S} = \RD \times \left( -1 , 0 \right)$. We introduce the diffeomorphism $\Sigma$,

\begin{equation}\label{diffeo}
\Sigma := \begin{array}{l}
\quad S \quad \rightarrow \quad \quad \quad \Omega \\
(X,z) \mapsto \left( X,z + \sigma(X,z) \right),
\end{array}
\end{equation}

\noindent where

\begin{equation*}
\sigma(X,z) := z \left( \epsilon \zeta(X) - \beta b(X) \right) + \epsilon \zeta(X).
\end{equation*}

\noindent In the following, we will focus on the bottom contribution and we refer to \cite{Castro_Lannes_vorticity} for the other terms. We keep the notations of \cite{Castro_Lannes_vorticity}. We define

\begin{equation*}
\text{U}^{\sigma,\mu} \left[ \epsilon \zeta, \beta b \right] \left(\psi, \omega \right) = \text{U}^{\mu} = \begin{pmatrix} \sqrt{\mu} \text{V} \\ \text{w} \end{pmatrix} = \textbf{U}^{\mu} \circ \Sigma  \text{, } \omega = \bm{\omega} \circ \Sigma,
\end{equation*}

\noindent and

\begin{equation*}
\nabla^{\sigma,\mu}_{\! X,z} = \left( J_{\Sigma}^{-1} \right)^{t} \nabla_{\! X,z}^{\mu} \text{, where } \left( J_{\Sigma}^{-1} \right)^{t} = \begin{pmatrix} Id_{d \times d} \;  \frac{-\sqrt{\mu}  \nabla \sigma}{1+\partial_{z} \sigma} \\ \quad 0 \quad \quad \frac{1}{1 + \partial_{z} \sigma} \end{pmatrix}.
\end{equation*}

\noindent Furthermore, for $\text{A} = \textbf{A} \circ \Sigma$,

\begin{equation*}
\text{curl}^{\sigma,\mu} \text{A} = \left( \text{curl}^{\mu} \;  \textbf{A} \right) \circ \Sigma = \nabla^{\sigma,\mu}_{\! X,z} \times \text{A} \text{, } \text{div}^{\sigma,\mu} \text{A} = \left( \text{div}^{\mu} \;  \textbf{A} \right) \circ \Sigma = \nabla^{\sigma,\mu}_{\! X,z} \cdot \text{A}.
\end{equation*}

\noindent Finally, if $\text{A}$ is vector field on $\mathcal{S}$, 

\begin{equation*}
\underline{\text{A}} = \text{A}_{|z=0} \text{, } \text{A}_{b} = \text{A}_{|z=-1} \text{ and } \text{A}_{\sslash} = \frac{1}{\sqrt{\mu}} \underline{\text{A}}_{h} + \epsilon \underline{\text{A}}_{v} \nabla \zeta. 
\end{equation*}

\noindent Then, $\text{U}^{\mu}$ is the unique solution in $H^{1}(\mathcal{S})$ of

\begin{equation}\label{div_curl_formulation_S}
\left\{
\begin{aligned}
&\text{curl}^{\sigma,\mu} \; \text{U}^{\mu} = \mu \, \omega \text{ in } \mathcal{S},\\
&\text{div}^{\sigma,\mu} \; \text{U}^{\mu} = 0 \text{ in } \mathcal{S},\\
&\Us^{\mu} = \nabla \psi + \frac{\nabla^{\perp}}{\Delta} \left( \underline{\omega} \cdot N^{\mu} \right) \text{ on } \left\{ z=0 \right\},\\
&\text{U}_{b}^{\mu} \cdot N^{\mu}_{b} = 0 \text{ on } \left\{ z=-1 \right\}.
\end{aligned}
\right.
\end{equation}

\noindent We also keep the notations in \cite{masmoudi_rousset}. If $\text{A} = \textbf{A} \circ \Sigma$, we define

\small
\begin{equation*}
\partial_{i}^{\sigma} \text{A} = \partial_{i} \textbf{A} \circ \Sigma \text{, } i \in \left\{t,x,y,z \right\} \text{, } \partial_{i}^{\sigma} = \partial_{i} - \frac{\partial_{i} \sigma}{1+\partial_{z} \sigma} \partial_{z} \text{, } i \in \left\{x,y,t \right\} \text{ and } \partial_{z}^{\sigma} = \frac{1}{1+\partial_{z} \sigma} \partial_{z}.
\end{equation*}
\normalsize

\noindent Then, by a change of variables and Proposition \ref{formulation_nabla} we get the following variational formulation for $\text{U}^{\mu}$. For all $\text{C} \in H^{1}(\mathcal{S})$,

\small
\begin{equation}\label{variational_formulation_U}
\int_{\mathcal{S}} \nabla^{\mu}_{\! X,z} \text{U}^{\mu} \cdot P \left( \Sigma \right) \nabla^{\mu}_{\! X,z} \text{C} = \mu \int_{\mathcal{S}} \left( 1+\partial_{z} \sigma \right) \omega \cdot \text{curl}^{\sigma, \mu} \text{C} + \int_{\RD} l^{\mu}[ \epsilon \zeta] \! \left( \underline{\text{U}}^{\mu} \right) \cdot \underline{C} - \int_{\RD} l^{\mu}[ \beta b]\! \left( \text{U}^{\mu}_{b} \right) \cdot C_{b},
\end{equation}
\normalsize

\noindent where $P \left( \Sigma \right) = \left( 1+\partial_{z} \sigma \right) J_{\Sigma}^{-1} \left( J_{\Sigma}^{-1} \right)^{t}$ and

\begin{equation*}
l^{\mu}[\eta] \! \left(\text{U}^{\mu}_{|z=\eta} \right) = \begin{pmatrix} \sqrt{\mu} \nabla \text{w}_{|z=\eta} - \mu^{\frac{3}{2}} \left( \nabla^{\perp} \eta \cdot \nabla \right) \text{V}^{\perp}_{|z=\eta} \\ - \mu \nabla \cdot \text{V}_{|z=\eta} \end{pmatrix}.
\end{equation*}

\noindent In order to obtain higher order estimates on $\text{U}^{\mu}$, we have to separate the regularity on $z$ and the regularity on $X$. We use the following spaces.

\begin{definition}\label{Hsk_spaces}

\noindent We define the spaces \upshape $H^{s,k}$ \itshape

\upshape
\begin{equation*}
H^{s,k} \! \left( \mathcal{S} \right) = \underset{0 \leq l \leq k}{\bigcap} H^{l}_{\! z} \! \left( -1,0 \; ; H^{s-l}_{\! X} \left( \RD \right) \right) \text{ and } \left\lvert u \right\rvert_{H^{s,k}} = \underset{0 \leq l \leq k}{\sum} \left\lvert \Lambda^{s-j} \partial_{z}^{j} u \right\rvert_{2}.
\end{equation*}
\itshape

\end{definition}

\noindent Furthermore, if $\alpha \in \mathbb{N}^{d} \backslash \{0\}$, we define the \textit{Alinhac's good unknown}

\begin{equation}\label{alinhac_psi}
\psia = \partial^{\alpha} \psi - \epsilon \underline{\text{w}} \partial^{\alpha} \zeta \text{ and } \psi_{(0)} = \psi.
\end{equation}

\noindent This quantities play an important role in the wellposedness of the water waves equations (see \cite{Alinhac_good_unknown_ww} or \cite{Lannes_ww}). In fact, more generally, if  $\text{A}$ is vector field on $\mathcal{S}$, we denote by 

\begin{equation}\label{alinhac_U}
\text{A}_{(\alpha)} = \partial^{\alpha} \text{A} - \partial^{\alpha} \sigma \partial_{z}^{\sigma} \text{A} \text{ , } \text{A}_{(0)} = \text{A} \text{ , }  \text{\underline{A}}_{(\alpha)} = \partial^{\alpha} \text{\underline{A}} -\epsilon \partial^{\alpha} \zeta \underline{\partial_{z}^{\sigma} \text{A}} \text{ and } \text{\underline{A}}_{(0)} = \text{\underline{A}}.
\end{equation}

\noindent We can now give high order estimates on $\text{U}^{\mu}$. We recall that $M_{N}$ is defined in \eqref{M_N_def}.

\begin{thm}\label{high_order_estimate_U}
\noindent Let \upshape $N \in \mathbb{N}$, $N \geq 5$. \itshape Then, under the assumptions of Theorem \ref{control_U}, for all \upshape $0 \leq l \leq 1$ \itshape and \upshape $0 \leq l \leq k \leq N-1$, \itshape the straightened velocity \upshape $\text{U}^{\mu}$, \itshape satisfies

\upshape
\begin{equation*}
\llver \nabla^{\mu}_{\! X,z} \text{U}^{\mu} \rrver_{H^{k,l}} \leq \mu M_{N} \left( \left\lvert \mathfrak{P} \psi \right\rvert_{H^{1}} + \underset{1 < |\alpha| \leq k+1}{\sum} \left\lvert \mathfrak{P} \psia \right\rvert_{2} + \llver  \omega \rrver_{H^{k,l}} + \left\lvert \frac{\Lambda^{k}}{\mathfrak{P}} \left( \omega_{b} \cdot N^{\mu}_{b} \right) \right\rvert_{2} \right).
\end{equation*}
\itshape
\end{thm}

\begin{proof}
\noindent We start with $l=0$. We follow the proof of Proposition 3.12 and Proposition 5.8 in \cite{Castro_Lannes_vorticity}. Let $k \in [1,N-1]$, $\alpha \in \mathbb{N}^{d}$ with $|\alpha| \leq k$. We take $C= \partial^{2 \alpha} \text{U}^{\mu}$ in \eqref{variational_formulation_U}\footnote{A. Castro and D. Lannes explain why we can take such a $\text{C}$ in the variational formulation.} and we get

\begin{align*}
\int_{\mathcal{S}} \hspace{-0.1cm} \nabla^{\mu}_{\! X,z} \text{U}^{\mu} \cdot P \left( \Sigma \right) \nabla^{\mu}_{\! X,z} \partial^{2 \alpha} \text{U}^{\mu} = \mu \hspace{-0.1cm} \int_{\mathcal{S}} \hspace{-0.1cm} \left( 1+\partial_{z} \sigma \right) \omega \cdot \text{curl}^{\sigma, \mu} \partial^{2 \alpha} \text{U}^{\mu} &+ \int_{\RD} \hspace{-0.3cm} l^{\mu}[ \epsilon \zeta] \! \left( \underline{\text{U}}^{\mu} \right) \cdot \partial^{2 \alpha} \underline{\text{U}}^{\mu}\\
& - \int_{\RD} \hspace{-0.3cm} l^{\mu}[ \beta b]\! \left( \text{U}^{\mu}_{b} \right) \cdot \partial^{2 \alpha} \text{U}^{\mu}_{b}.
\end{align*}

\noindent We focus on the bottom contribution, which is the last term of the previous equation. Using the fact that $\text{w}_{b} = \mu  \beta \nabla b \cdot \text{V}_{b}$, we have

\small
\begin{equation*}
\begin{aligned}
(-1)^{|\alpha|} \int_{\RD} l^{\mu}[\beta b]\left( \text{U}_{b} \right) \cdot \partial^{2 \alpha} \text{U}_{b} &= \int_{\RD} \hspace{-0.3cm} 2 \mu \partial^{\alpha} \nabla \text{w}_{b} \cdot \partial^{\alpha} \text{V}_{b} - \mu^{2} \beta \partial^{\alpha} \left[\left(\nabla^{\perp} b \cdot \nabla \right) \text{V}_{b}^{\perp} \right] \cdot \partial^{\alpha} \text{V}_{b}\\
&= \int_{\RD} \hspace{-0.3cm} 2 \mu^{2} \beta \partial^{\alpha} \nabla \left(\nabla b \cdot \text{V}_{b} \right) \cdot \partial^{\alpha} \text{V}_{b} - \mu^{2} \beta \partial^{\alpha} \left[\left(\nabla^{\perp} b \cdot \nabla \right) \text{V}_{b}^{\perp} \right] \cdot \partial^{\alpha} \text{V}_{b}\\
&= \underbrace{\int_{\RD} \hspace{-0.3cm} 2 \mu^{2} \beta \left(\nabla b\right)^{t} \cdot \partial^{\alpha} \nabla \text{V}_{b} \cdot \partial^{\alpha} \text{V}_{b} - \beta \mu^{2} \left[\left(\nabla^{\perp} b \cdot \nabla \right) \partial^{\alpha} \text{V}_{b}^{\perp} \right] \cdot \partial^{\alpha} \text{V}_{b}}_{I_{1}}\\
&\;\;+ \underbrace{\int_{\RD} \hspace{-0.3cm} 2 \mu^{2} \beta \left[\partial^{\alpha} \nabla, \nabla b \right] \text{V}_{b} \cdot \partial^{\alpha} \text{V}_{b} - \beta \mu^{2} \left[ \partial^{\alpha}, \left( \nabla^{\perp} b \cdot \nabla \right) \right] \text{V}_{b}^{\perp} \cdot \partial^{\alpha} \text{V}_{b}}_{I_{2}}.
\end{aligned}
\end{equation*}
\normalsize

\noindent Then, a careful computation gives 

\begin{align*}
\left\lvert I_{1}  \right\rvert &= \left\lvert \mu^{2} \beta \int_{\RD} \partial^{2}_{x} b \left( \partial^{\alpha} \text{V}_{bx} \right)^{2} + \partial^{2}_{y} b \left( \partial^{\alpha} \text{V}_{by} \right)^{2} + 2 \mu^{2} \beta \int_{\RD} \partial_{xy}^{2} b \; \partial^{\alpha} \text{V}_{bx} \; \partial^{\alpha} \text{V}_{by} \right\rvert\\
&\leq  \mu C \left(\delta, \frac{1}{h_{\min}}, \epsilon \left\lvert \zeta \right\rvert_{W^{1,\infty}}, \beta \left\lvert b \right\rvert_{W^{2,\infty}} \right) \llver \partial^{\alpha} \text{U}^{\mu} \rrver_{2}^{2} + \delta \llver \nabla^{\mu}_{\! X,z} \partial^{\alpha} \text{U}^{\mu} \rrver_{2}^{2}\\
&\leq  C \left(\delta, \frac{1}{h_{\min}}, \epsilon \left\lvert \zeta \right\rvert_{W^{1,\infty}}, \beta \left\lvert b \right\rvert_{W^{2,\infty}} \right) \llver \nabla^{\mu}_{\! X,z} \text{U}^{\mu} \rrver_{H^{k-1}}^{2} + \delta \llver \nabla^{\mu}_{\! X,z} \partial^{\alpha} \text{U}^{\mu} \rrver_{2}^{2},
\end{align*}

\noindent where $\delta>0$ is small enough and where we use the following Lemma.

\begin{lemma}\label{trace_lemma}
\noindent Let \upshape $\zeta \text{, } b \in W^{1,\infty} \hspace{-0.05cm} \left( \RD \right)$ \itshape, such that Condition \eqref{nonvanishing} is satisfied. Then, for all \upshape $u \in H^{1}\left( \mathcal{S} \right)$ and $\delta > 0$,

\begin{equation*}
\left\lvert \underline{u} \right\rvert_{2}^{2} + \left\lvert u_{b} \right\rvert_{2}^{2} \leq C \left(\delta, \frac{1}{h_{\min}}, \epsilon \left\lvert \zeta \right\rvert_{W^{1,\infty}}, \beta \left\lvert b \right\rvert_{W^{1,\infty}} \right) \llver u \rrver_{2}^{2} + \delta \llver \partial_{z} u \rrver_{2}^{2}.
\end{equation*}
\itshape

\end{lemma}

\noindent Furthermore, using Lemma \ref{commutator_estimate} and the previous Lemma, we get

\begin{align*}
\left\lvert I_{2}  \right\rvert &\leq C \mu \beta \left\lvert \nabla b \right\rvert_{H^{k+1}} \left\lvert \text{U}^{\mu}_{b} \right\rvert_{H^{k}} \left\lvert \partial^{\alpha} \text{U}^{\mu}_{b} \right\rvert_{2}\\
&\leq  \mu C \left(\delta, \frac{1}{h_{\min}}, \epsilon \left\lvert \zeta \right\rvert_{W^{1,\infty}}, \beta \left\lvert b \right\rvert_{W^{1,\infty}}, \beta \left\lvert \nabla b \right\rvert_{H^{k+1}} \right) \llver \nabla^{\mu}_{\! X,z} \text{U}^{\mu} \rrver_{H^{k-1}}^{2} + \delta \llver \nabla^{\mu}_{\! X,z} \partial^{\alpha} \text{U}^{\mu} \rrver_{2}^{2}.
\end{align*}

\noindent For the surface contribution, we can do the same thing as in Proposition 3.12 and Proposition 5.8 in \cite{Castro_Lannes_vorticity}, using the previous Lemma to control $\partial^{\alpha} \text{w}$. Finally, for the other terms, the main idea is the following Lemma (which is a small adaptation of Lemma 3.13 and Lemma 5.6 in \cite{Castro_Lannes_vorticity}).

\begin{lemma}\label{psit_control}
\noindent Let \upshape $\psit$ \itshape the unique solution of \itshape $\Delta \psit = \underline{\omega} \cdot N^{\mu}$ \itshape in $\Hdot^{1}(\RD)$. Under the assumptions of the Theorem, we have the following estimate

\upshape
\begin{equation*}
\left\lvert \mathfrak{P} \nabla^{\perp} \psit \right\rvert_{H^{k}} \leq M_{N} \left(  \llver  \omega \rrver_{H^{k,0}} + \left\lvert \frac{\Lambda^{k}}{\mathfrak{P}} \left( \omega_{b} \cdot N^{\mu}_{b} \right) \right\rvert_{2} \right).
\end{equation*}
\itshape

\end{lemma}

\noindent Gathering the previous estimates with the estimate without the bottom contribution in Proposition 5.8 in \cite{Castro_Lannes_vorticity}, we get

\small
\begin{equation*}
\llver \partial^{\alpha} \nabla^{\mu} \text{U}^{\mu} \rrver_{2} \hspace{-0.1cm} \leq \hspace{-0.1cm} \mu M_{N} \hspace{-0.15cm} \left( \hspace{-0.15cm} \left\lvert \mathfrak{P} \psi \right\rvert_{H^{1}} \hspace{-0.1cm} + \hspace{-0.65cm} \underset{1 < |\alpha| \leq k+1}{\sum} \hspace{-0.55cm} \left\lvert \mathfrak{P} \psia \right\rvert \hspace{-0.1cm} + \hspace{-0.1cm} \llver  \omega \rrver_{H^{k,0}} \hspace{-0.1cm} + \hspace{-0.1cm} \left\lvert \hspace{-0.05cm} \frac{\Lambda^{k}}{\mathfrak{P}} \hspace{-0.1cm} \left( \omega_{b} \cdot N^{\mu}_{b} \right) \hspace{-0.05cm} \right\rvert_{2} \hspace{-0.1cm} \right)\hspace{-0.2cm} + \hspace{-0.1cm} M_{N} \hspace{-0.1cm} \llver \Lambda^{k-1} \nabla_{\! X,z} ^{\mu}\text{U}^{\mu} \rrver_{2} \hspace{-0.1cm} ,
\end{equation*}
\normalsize

\noindent and the inequality follows by a finite induction on $k$. If $l=1$, we can adapt the proof of Corollary 3.14 in \cite{Castro_Lannes_vorticity} easily.
\end{proof}

\noindent 

\begin{remark}\label{control_omega_bott}
\noindent Notice that for $k \geq 2$, we have

\begin{equation*}
\left\lvert \frac{\Lambda^{k}}{\mathfrak{P}} \left( \omega_{b} \cdot N^{\mu}_{b} \right) \right\rvert_{2} \leq C \left(\frac{1}{h_{\min}}, \mu_{\max}, \beta \left\vert \nabla b \right\rvert_{H^{k+1}} \right) \left( \llver \omega \rrver_{H^{k,1}} + \left\lvert \frac{1}{\mathfrak{P}} \left( \omega_{b} \cdot N^{\mu}_{b} \right) \right\rvert_{2} \right),
\end{equation*}

\noindent thanks to Lemma \ref{H_ast_reg}, Lemma \ref{boundary_control} and Lemma \ref{product_estimate}.
\end{remark}

\subsection{Time derivatives and few remarks about the good unknown}

\noindent This part is devoted to recall and adapt some results in \cite{Castro_Lannes_vorticity}. Unlike the previous Propositions, adding a non flat bottom is not painful. That is why we do not give proofs. We refer to section 3.5 and 3.6 in \cite{Castro_Lannes_vorticity} for the details. Firstly, in order to obtain an energy estimate of the Castro-Lannes water waves formulation, we need to control $\partial_{t} \text{U}^{\mu}$. This is the purpose of the following result.

\begin{prop}\label{time_derivative_control}
\noindent Let \upshape $T > 0$, $\zeta \in \mathcal{C}^{1} \left([0,T], W^{2,\infty} \left( \RD \right) \right)$, $b \in W^{2,\infty} \left( \RD \right)$ \itshape such that \eqref{nonvanishing} is satisfied for $0 \leq t \leq T$, \upshape $\psi \in \mathcal{C}^{1}\left([0,T], \Hdot^{\frac{3}{2}} \left( \RD \right) \right)$ \itshape and \upshape $\omega \in \mathcal{C}^{1} \left([0,T], L^{2} \left( \mathcal{S} \right)^{d+1} \right)$ \itshape such that $\nabla^{\mu,\sigma}_{\! X,z} \cdot \omega = 0$ for $0 \leq t \leq T$. Then,

\upshape
\begin{align*}
\partial_{t} \left( \text{U}^{\sigma,\mu} [\epsilon \zeta, \beta b] \left( \psi, \omega \right) \right) = \text{U}^{\sigma,\mu} [\epsilon \zeta, \beta b] \Big( \partial_{t} \psi - \epsilon \underline{\text{w}} \partial_{t} \zeta + \epsilon &\sqrt{\mu} \frac{\nabla}{\Delta} \cdot \left( \underline{\omega_{h}}^{\perp} \partial_{t} \zeta \right), \partial_{t}^{\sigma} \omega \Big)\\
& + \partial_{t} \sigma \partial_{z}^{\sigma} \left( \text{U}^{\mu,\sigma}[\epsilon \zeta, \beta b] \left( \psi, \omega \right) \right).
\end{align*}
\itshape

\noindent Furthermore, for \upshape$N \geq 5$\itshape , \upshape$\text{U}^{\mu} = \text{U}^{\sigma,\mu} [\epsilon \zeta, \beta b]$ \itshape satisfies

\upshape
\begin{align*}
\sqrt{\mu} \llver \partial_{t} \text{U}^{\mu} \rrver_{2} +  \llver \partial_{t} \nabla^{\mu}_{\! X,z} \text{U}^{\mu} \rrver_{H^{N-2,0}} &\leq \mu \max \left( M_{N}, \epsilon \left\lvert \partial_{t} \zeta \right\rvert_{H^{N-1}} \right) \times \\
&\hspace{-2.5cm} \Bigg(\left\lvert \mathfrak{P} \partial_{t} \psi \right\rvert_{H^{1}} + \hspace{-0.5cm} \underset{1 < |\alpha| \leq N-1}{\sum} \hspace{-0.5cm} \left\lvert \mathfrak{P} \partial_{t} \psia \right\rvert_{2} + \llver \partial_{t} \omega \rrver_{H^{N-2,0}} + \left\lvert \frac{\Lambda^{N-2}}{\mathfrak{P}} \left( \partial_{t} \omega_{b} \cdot N^{\mu}_{b} \right) \right\rvert_{2}\\
& \hspace{-1.5cm} + \left\lvert \mathfrak{P} \psi \right\rvert_{H^{1}} + \hspace{-0.4cm} \underset{1 < |\alpha| \leq N}{\sum} \left\lvert \mathfrak{P} \psia \right\rvert_{2} + \llver  \omega \rrver_{H^{N-1,1}} + \left\lvert \frac{1}{\mathfrak{P}} \left( \omega_{b} \cdot N^{\mu}_{b} \right) \right\rvert_{2} \Bigg).
\end{align*}
\itshape
\end{prop}

\noindent Secondly, in the context of water waves, the \textit{Alinhac's good unknowns} play a crucial role. N. Masmoudi and F. Rousset remarked in \cite{masmoudi_rousset} that the \textit{Alinhac's good unknown} $\Umua$ is almost incompressible and A. Castro and D. Lannes showed that the $\text{curl}^{\sigma,\mu}$ of $\text{U}^{\mu}_{(\alpha)}$ is also well controlled. This is the purpose of the following Proposition. We recall that $\Umua$ is defined in \eqref{alinhac_U}.

\begin{prop}\label{div_curl_good_unknowns_control}
\noindent Let \upshape$N \geq 5$\itshape, \upshape$\zeta \in H^{N}(\RD)$\itshape, \upshape$b \in L^{\infty} \cap \Hdot^{N+1}(\RD)$ \itshape such that Condition \eqref{nonvanishing} is satisfied and \upshape$\omega \in H^{N-1}(\mathcal{S})$ \itshape  such that \upshape$\nabla^{\sigma,\mu} \cdot \omega = 0$\itshape. Then if we denote by \upshape$\text{U}^{\mu} = \text{U}^{\mu, \sigma} [\epsilon \zeta, \beta b]$\itshape, we have for \upshape$1 \leq |\alpha| \leq N$\itshape,

\upshape
\begin{align*}
&\llver \nabla^{\sigma,\mu}_{\! X,z} \cdot \Umua \rrver_{2} + \llver \nabla^{\sigma,\mu}_{\! X,z} \times \Umua  - \mu \partial^{\alpha} \omega \rrver_{2}\\
&\hspace{1cm} \leq \mu \lver \left( \epsilon \zeta, \beta b \right) \rver_{H^{N}} M_{N} \hspace{-0.1cm} \left(\hspace{-0.1cm} \left\lvert \mathfrak{P} \psi \right\rvert_{H^{1}} \hspace{-0.1cm} + \hspace{-0.5cm} \underset{1 < |\alpha'| \leq |\alpha|}{\sum} \hspace{-0.5cm} \left\lvert \mathfrak{P} \psi_{(\alpha^{'})} \right\rvert_{2} \hspace{-0.1cm} + \hspace{-0.05cm} \llver \omega \rrver_{H^{\max(|\alpha|-1,1)}} \hspace{-0.1cm} + \hspace{-0.05cm} \left\lvert \frac{1}{\mathfrak{P}} \left( \omega_{b} \cdot N_{b}^{\mu} \right) \right\rvert_{2} \right) \hspace{-0.1cm},
\end{align*}
\itshape

\noindent and

\upshape
\begin{equation*}
\left\lvert \mathfrak{P} \psia \right\rvert_{2} \leq M_{N} \left( \left\lvert \mathfrak{P} \psi \right\rvert_{H^{3}} + \frac{1}{\sqrt{\mu}} \underset{1 < |\alpha'| \leq |\alpha| - 1}{\sum} \hspace{-0.3cm}  \llver \nabla_{\! X} \text{U}^{\mu}_{(\alpha')} \rrver_{2} + \llver \omega \rrver_{H^{N-1}} +  \left\lvert \frac{1}{\mathfrak{P}} \left( \omega_{b} \cdot N_{b}^{\mu}  \right) \right\rvert_{2} \right).
\end{equation*}
\itshape
\end{prop}

\noindent Furthermore, we can control $\left\lvert \mathfrak{P} \psi \right\rvert_{H^{3}}$ by $\text{U}^{\mu}$ and $\omega$.

\begin{prop}\label{frakP_psi_H3}
\noindent Let \upshape$N \geq 5$\itshape, \upshape$\zeta \in H^{N}(\RD)$\itshape, \upshape$b \in L^{\infty} \cap \Hdot^{N+1}(\RD)$ \itshape such that Condition \eqref{nonvanishing} is satisfied and \upshape$\omega \in H^{2,1}(\mathcal{S})$ \itshape such that \upshape$\nabla^{\sigma,\mu} \cdot \omega = 0$\itshape. Then, 

\upshape
\begin{equation*}
\left\lvert \mathfrak{P} \psi \right\rvert_{H^{3}} \leq M_{N} \left(\frac{1}{\sqrt{\mu}} \llver \Lambda^{3} \text{U}^{\sigma, \mu}[\epsilon \zeta, \beta b]\left( \psi, \omega \right) \rrver_{2} + \llver \omega \rrver_{H^{2,1}} + \left\lvert \frac{1}{\mathfrak{P}} \left( \omega_{b} \cdot N^{\mu}_{b} \right) \right\rvert_{2} \right).
\end{equation*}
\itshape
\end{prop}

\begin{proof}
\noindent The proof is a small adaptation of Lemma 3.23 in \cite{Castro_Lannes_vorticity}. 
\end{proof}

\noindent Finally, we give a result that is useful for the energy estimate. Since the proof is a little different to Corollary 3.21 in \cite{Castro_Lannes_vorticity}, we give it. Notice that the main difference with Corollary 3.21 in \cite{Castro_Lannes_vorticity} is the fact that we do not have a flat bottom.

\begin{prop}\label{speed_surface_control}
\noindent Let \upshape$N \geq 5$\itshape, \upshape$\zeta \in H^{N}(\RD)$\itshape, \upshape$b \in L^{\infty} \cap \Hdot^{N+1}(\RD)$ \itshape and \upshape$\omega \in H^{N-1}(\mathcal{S})$ \itshape such that \upshape$\nabla^{\sigma,\mu} \cdot \omega = 0$\itshape. Then, for \upshape$k=x,y \text{, } |\gamma| \leq N-1 \text{, } \alpha \text{ }$\itshape such that \upshape$\partial^{\alpha} = \partial_{k} \partial^{\gamma} \text{ and } \varphi \in H^{\frac{1}{2}}(\RD)$, we have

\upshape
\begin{equation*}
\begin{aligned}
\left(\hspace{-0.1cm} \varphi \text{, } \hspace{-0.1cm} \frac{1}{\mu} \partial_{k} \underline{\text{U}}^{\mu}_{\left( \gamma \right)} \hspace{-0.1cm}\cdot\hspace{-0.1cm} N^{\mu} \hspace{-0.1cm}\right) \hspace{-0.15cm} \leq \hspace{-0.05cm} M_{N} \hspace{-0.1cm} \left( \hspace{-0.15cm} \left\lvert \mathfrak{P} \psi \right\rvert_{H^{1}} \hspace{-0.1cm} + \hspace{-0.6cm} \underset{1 < \left\lvert\alpha^{'}\right\rvert \leq |\alpha|}{\sum} \hspace{-0.5cm} \left\lvert \mathfrak{P} \psi_{(\alpha^{'})} \right\rvert_{2} \hspace{-0.2cm} + \hspace{-0.1cm} \llver \omega \rrver_{H^{|\alpha|-1}} \hspace{-0.15cm} + \hspace{-0.1cm} \left\lvert \hspace{-0.05cm} \frac{1}{\mathfrak{P}} \hspace{-0.05cm} \left( \omega_{b} \cdot N^{\mu} \hspace{-0.05cm} \right) \right\rvert_{2} \hspace{-0.1cm} \right) \hspace{-0.1cm}  \times&\\
&\hspace{-2.8cm} \left[\left\lvert \mathfrak{P} \varphi \right\rvert_{2} + \left\lvert \frac{1}{\sqrt{1+\sqrt{\mu}|D|}} \varphi \right\rvert_{2} \right] \hspace{-0.1cm},
\end{aligned}
\end{equation*}
\itshape

\noindent where we denote by \upshape$\text{U}^{\mu} = \text{U}^{\sigma,\mu} [\epsilon \zeta, \beta b]$\itshape.
\end{prop}

\begin{proof}
\noindent Notice that when $\gamma \neq 0$,

\begin{equation*}
\partial_{k} \text{U}^{\mu}_{(\gamma)} = \text{U}^{\mu}_{(\alpha)} - \partial^{\gamma} \sigma \partial_{k} \partial_{z}^{\sigma} \text{U}^{\mu}.
\end{equation*}

\noindent Then, using Lemma \ref{boundary_control}, it is easy to check that

\begin{equation*}
\left(\varphi, \underline{\partial^{\gamma} \sigma \partial_{k} \partial_{z}^{\sigma} \text{U}^{\mu}} \cdot N^{\mu} \right) \leq M_{N} \left\lvert \frac{1}{\sqrt{1+\sqrt{\mu}|D|}} \varphi \right\rvert_{2} \llver \nabla^{\mu}_{\! X,z} \text{U}^{\mu} \rrver_{H^{2}}.
\end{equation*}

\noindent Furthermore, using the Green identity we get

\begin{equation*}
\left(\varphi \text{,} \text{U}^{\mu}_{(\alpha)} \cdot N^{\mu} \right) \hspace{-0.15cm} = \hspace{-0.15cm} \int_{\mathcal{S}} \hspace{-0.15cm} \left(1+\partial_{z} \sigma \right) \varphi^{\dag} \nabla^{\sigma, \mu}_{\! X,z} \cdot \text{U}^{\mu}_{(\alpha)} + \int_{\mathcal{S}} \hspace{-0.15cm} \left(1+\partial_{z} \sigma \right) \text{U}^{\mu}_{(\alpha)} \cdot \nabla^{\sigma, \mu}_{\! X,z} \varphi^{\dag} + \left(\varphi^{\dag}_{b} \text{,} \left(\text{U}^{\mu}_{(\alpha)} \right)_{b} \cdot N^{\mu}_{b} \right),
\end{equation*}

\noindent where $\varphi^{\dag} =\chi \left(z \sqrt{\mu} |D| \right) \varphi$ and $\chi$ is an even positive compactly supported function equal to $1$ near $0$. Then, using the fact that $\text{U}^{\mu}_{b} \cdot N^{\mu}_{b} = 0$ and the trace Lemma, we get

\begin{align*}
\left(\varphi^{\dag}_{b} \text{ , } \left(\text{U}^{\mu}_{(\alpha)} \right)_{b} \cdot N^{\mu}_{b} \right) &= \left(\chi(\sqrt{\mu} |D|) \varphi \text{ , } \partial^{\alpha} \text{U}^{\mu}_{b} \cdot N^{\mu}_{b} - \beta \partial^{\alpha} b \left( \partial_{z}^{\sigma} \text{U}^{\mu} \right)_{b} \cdot N^{\mu}_{b} \right)\\
&= \left(\chi(\sqrt{\mu} |D|) \varphi \text{ , } \mu \beta \left[\nabla b \text{, } \partial^{\alpha} \right] \cdot \text {V}_{b} - \beta \partial^{\alpha} b \left( \partial_{z}^{\sigma} \text{U}^{\mu} \right)_{b} \cdot N^{\mu}_{b} \right)\\
&\leq M_{N} \left(\sqrt{\mu} \llver \text{U}^{\mu} \rrver_{H^{N}} + \llver \text{U}^{\mu} \rrver_{H^{2,2}} \right) \left\lvert \chi(\sqrt{\mu} |D|) \varphi \right\rvert_{2}.
\end{align*}

\noindent Therefore, using Proposition \ref{div_curl_good_unknowns_control}, Theorem \ref{high_order_estimate_U} and the following Lemma (Lemma 2.20 and Lemma 2.34 in \cite{Lannes_ww}) we get the control.

\begin{lemma}
\noindent Let $\varphi \in H^{\frac{1}{2}}(\RD)$ and $\chi$ an even positive compactly supported function equal to $1$ near $0$. Then,

\upshape
\begin{equation*}
\llver \chi \left(z \sqrt{\mu} |D| \right) \varphi \rrver_{2} \leq C \left\lvert \frac{1}{\sqrt{1+\sqrt{\mu}|D|}} \varphi \right\rvert_{2} \text{ and } \llver \nabla^{\mu}_{\! X,z} \left(\chi \left(z \sqrt{\mu} |D| \right) \varphi \right) \rrver_{2} \leq C \sqrt{\mu} \left\lvert \mathfrak{P} \varphi \right\rvert_{2}.
\end{equation*}
\itshape
\end{lemma}
\end{proof}

\section{Well-posedness of the water waves equations}

\subsection{Framework}\label{framework_ww}

\noindent In this section, we prove a local well-posedness result of the water waves equations. We improve the result of \cite{Castro_Lannes_vorticity} by adding a non flat bottom, a non constant pressure at the surface and a Coriolis forcing. In order to work on a fixed domain, we seek a system of $3$ equations on $\zeta$, $\psi$ and $\omega = \bm{\omega} \circ \Sigma$. We keep the first and the second equations of the Castro-Lannes formulation \eqref{Castro_lannes_formulation}. It is easy to check that $\omega$ satisfies

\begin{equation}\label{straight_vorticity_eq}
\partial_{t}^{\sigma} \omega + \frac{\epsilon}{\mu} \left( \text{U}^{\mu} \cdot \nabla_{\! X,z}^{\sigma,\mu}  \right) \omega \hspace{-0.05cm} = \hspace{-0.05cm} \frac{\epsilon}{\mu} \left( \omega \cdot \nabla^{\sigma, \mu}_{\! X,z} \right) \text{U}^{\mu}  + \frac{\epsilon}{\mu \text{Ro}} \partial_{z}^{\sigma} \text{U}^{\mu},
\end{equation}

\noindent where $\text{U}^{\mu} = \text{U}^{\sigma,\mu}[\epsilon \zeta, \beta b]$. Then, in the following the water waves equations will be the system

\small
\begin{equation}\label{Castro_lannes_formulation_straight}
\left\{
\begin{aligned}
&\hspace{-0.05cm} \partial_{t} \zeta - \frac{1}{\mu} \underline{\text{U}}^{\mu} \cdot N^{\mu} = 0,\\
&\hspace{-0.05cm} \partial_{t} \psi \hspace{-0.05cm} + \hspace{-0.05cm} \zeta \hspace{-0.05cm} + \hspace{-0.05cm} \frac{\epsilon}{2} \hspace{-0.05cm} \left\lvert  \text{U}^{\mu}_{\! \sslash} \right\rvert^{2} \hspace{-0.15cm} - \hspace{-0.05cm} \frac{\epsilon}{2 \mu} \hspace{-0.05cm} \left(\hspace{-0.05cm} 1 + \hspace{-0.05cm} \epsilon^{2} \mu \left\lvert \nabla \zeta \right\rvert^{2} \hspace{-0.05cm} \right) \! \underline{\text{w}}^{2} \hspace{-0.1cm} + \hspace{-0.1cm} \epsilon \frac{\nabla}{\Delta} \hspace{-0.1cm} \cdot \hspace{-0.1cm} \left[ \left( \underline{\omega} \hspace{-0.05cm} \cdot \hspace{-0.05cm} N^{\mu} \hspace{0.05cm} + \frac{1}{\text{Ro}} \right) \underline{\text{V}}^{\perp} \right] = - P,\\
&\hspace{-0.05cm} \partial_{t}^{\sigma} \omega + \frac{\epsilon}{\mu} \left( \text{U}^{\mu} \cdot \nabla_{\! X,z}^{\sigma,\mu} \right) \omega \hspace{-0.05cm} = \hspace{-0.05cm} \frac{\epsilon}{\mu} \left( \omega \cdot \nabla^{\sigma, \mu}_{\! X,z} \right) \text{U}^{\mu}  + \frac{\epsilon}{\mu \text{Ro}} \partial_{z}^{\sigma} \text{U}^{\mu}.
\end{aligned}
\right.
\end{equation}
\normalsize

\noindent The following quantity is the energy that we will use to get the local wellposedness

\begin{equation*}
\mathcal{E}^{N} \left( \zeta, \psi, \omega \right) = \frac{1}{2} \left\lvert \zeta \right\rvert_{H^{N}}^{2} + \frac{1}{2} \left\lvert \mathfrak{P} \psi \right\rvert^{2}_{H^{3}} + \frac{1}{2} \underset{1 \leq |\alpha| \leq N}{\sum} \left\lvert \mathfrak{P} \psia \right\rvert_{2}^{2} + \frac{1}{2} \llver \omega \rrver_{H^{N-1}}^{2} +  \frac{1}{2} \left\lvert \frac{1}{\mathfrak{P}} \left( \omega_{b} \cdot N_{b}^{\mu} \right) \right\rvert_{2}^{2},
\end{equation*} 

\noindent where we recall that $\psia$ is given by \eqref{alinhac_psi}. For $T \geq 0$, we also introduce the energy space

\begin{equation*}
\text{E}^{N}_{T} = \left\lbrace \left( \zeta, \psi, \omega \right) \in \mathcal{C} \left( [0,T], H^{2}(\RD) \times \Hdot^{2}(\RD) \times H^{2}(\mathcal{S}) \right), \mathcal{E}^{N} \hspace{-0.1cm} \left( \zeta, \psi, \omega \right) \in L^{\infty}([0,T]) \right\rbrace.
\end{equation*} 

\noindent We also recall that $M_{N}$ is defined in \eqref{M_N_def}. We keep the organization of the section 4 in \cite{Castro_Lannes_vorticity}. First, we give an a priori estimate for the vorticity. Then, we explain briefly how we can quasilinearize the system and how we obtain a priori estimates for the full system. The last part of this section is devoted to the proof of the main result.

\subsection{A priori estimate for the vorticity}

\noindent In this part, we give a priori estimate for the straightened equation of the vorticity.

\begin{prop}\label{vorticity_estimate}
\noindent Let \upshape$N \geq 5$, $T > 0$, $b \in L^{\infty} \cap \Hdot^{N+1}(\RD) $\itshape and \upshape$(\zeta, \psi, \omega)\in \text{E}^{N}_{T}$ \itshape such that \eqref{straight_vorticity_eq} and Condition \eqref{nonvanishing} hold on \upshape$\left[0, T \right]$\itshape. We also assume that on \upshape$\left[0,T \right]$\itshape

\upshape
\begin{equation*}
\partial_{t} \zeta - \frac{1}{\mu} \text{U}^{\sigma,\mu} [\epsilon \zeta, \beta b] \cdot N^{\mu} = 0.
\end{equation*}
\itshape 

\noindent Then, 

\upshape
\begin{equation*}
\frac{d}{dt} \left( \llver \omega \rrver_{H^{N-1}}^{2} + \left\lvert \frac{1}{\mathfrak{P}} \left( \omega_{b} \cdot N_{b}^{\mu} \right) \right\rvert_{2}^{2} \right) \leq M_{N} \left( \epsilon \mathcal{E}^{N} \hspace{-0.1cm} \left( \zeta, \psi, \omega \right)^{\frac{3}{2}} + \max \left(\epsilon, \frac{\epsilon}{\text{Ro}} \right) \mathcal{E}^{N} \hspace{-0.1cm} \left( \zeta, \psi, \omega \right) \right).
\end{equation*}
\itshape
\end{prop}

\begin{proof}
\noindent We denote $\text{U}^{\sigma,\mu} [\epsilon \zeta, \beta b] = \text{U}^{\mu} = \begin{pmatrix} \sqrt{\mu} \text{V} \\ \text{w} \end{pmatrix}$. We can reformulate Equation \eqref{straight_vorticity_eq} as

\begin{equation*}
\partial_{t} \omega + \epsilon \left( \text{V} \cdot \nabla_{\! X} \right) \omega + \frac{\epsilon}{\mu} a \partial_{z} \omega = \frac{\epsilon}{\mu} \left( \omega \cdot \nabla^{\sigma, \mu}_{\! X,z} \right) \text{U}^{\mu} + \frac{\epsilon}{\mu \text{Ro}} \partial_{z}^{\sigma} \text{U}^{\mu},
\end{equation*}

\noindent where 

\begin{equation*}
a = \frac{1}{1+\partial_{z} \sigma} \left( \text{U}^{\mu} \cdot \begin{pmatrix} - \sqrt{\mu} \nabla_{\! X} \sigma \\ 1 \end{pmatrix} - (z+1) \underline{\text{U}}^{\mu} \cdot N^{\mu} \right).
\end{equation*}

\noindent Notice that $\underline{a} = a_{b} = 0$. Then, we get

\begin{equation*}
\partial_{t} \llver \omega \rrver_{2}^{2} = \epsilon \int_{S} \left( \nabla_{\! X} \cdot \text{V} + \frac{1}{\mu} \partial_{z} a \right) \omega^{2} + \frac{2}{\mu} \left(\omega \cdot \nabla^{\sigma,\mu}_{\! X,z} \right) \text{U}^{\mu} \cdot \omega + \frac{1}{\text{Ro}} \partial_{z}^{\sigma} \text{U}^{\mu} \cdot \omega,
\end{equation*}

\noindent and

\begin{equation*}
\begin{aligned}
\partial_{t} \llver \omega \rrver_{2}^{2} \leq \hspace{-0.05cm} \frac{\epsilon}{\mu} C \hspace{-0.1cm} \left( \frac{1}{h_{\min}}, \epsilon \left\lvert \zeta \right\rvert_{W^{1,\infty}} \hspace{-0.05cm} , \beta \left\lvert b \right\rvert_{W^{1,\infty}} \hspace{-0.1cm} \right) \hspace{-0.1cm} &\Big( \left[ \llver \nabla^{\mu}_{\! X,z} \text{U}^{\mu} \rrver_{\infty} \hspace{-0.3cm} + \sqrt{\mu} \llver \text{U}^{\mu} \rrver_{\infty} \right] \llver \omega \rrver_{2}^{2} \\
&\hspace{3cm} + \frac{1}{\text{Ro}} \llver \nabla^{\mu}_{\! X,z} \text{U}^{\mu} \rrver_{\infty} \llver \omega \rrver_{2} \Big) \hspace{-0.05cm},
\end{aligned}
\end{equation*}

\noindent where we use the fact that

\begin{equation*}
\left\lvert \underline{\text{U}}^{\mu} \cdot N^{\mu} \right\rvert_{L^{\infty}} \leq C \left(\epsilon \left\lvert \zeta \right\rvert_{W^{1,\infty}},\beta \left\lvert b \right\rvert_{W^{1,\infty}} \right) \left(\llver \partial_{z} \text{U}^{\mu} \rrver_{\infty} + \sqrt{\mu} \llver \text{U}^{\mu} \rrver_{\infty} \right).
\end{equation*}

\noindent The estimate for the $L^{2}$-norm of $\omega$ follows thanks to Theorem, \ref{control_U}, Theorem \ref{high_order_estimate_U} and Remark \ref{control_omega_bott}. For the high order estimates, we differentiate Equation \eqref{straight_vorticity_eq} and we easily obtain the control thanks to Theorem \ref{high_order_estimate_U} and Remark \ref{control_omega_bott} (see the proof of Proposition 4.1 in \cite{Castro_Lannes_vorticity}). Finally, taking the trace at the bottom of the vorticity equation in System \eqref{Castro_lannes_formulation}, we get the following equation for $\omega_{b} \cdot N_{b}^{\mu}$,

\begin{equation}
\partial_{t} \left( \omega_{b} \cdot N_{b}^{\mu} \right) + \epsilon \nabla \cdot \left( \left[\omega_{b} \cdot N_{b}^{\mu} + \frac{1}{\text{Ro}} \right] \text{V}_{b} \right) = 0,
\end{equation}

\noindent and then,

\begin{equation*}
\partial_{t} \left\lvert \frac{1}{\mathfrak{P}} \left( \omega_{b} \cdot N^{\mu}_{b} \right) \right\rvert_{2}^{2} \leq 2 \epsilon \left\lvert \sqrt{1+\sqrt{\mu} |D|} \left(\left[\omega_{b} \cdot N_{b}^{\mu} + \frac{1}{\text{Ro}} \right] \text{V}_{b} \right) \right\rvert_{2} \left\lvert \frac{1}{\mathfrak{P}} \left( \omega_{b} \cdot N^{\mu}_{b} \right) \right\rvert_{2}.
\end{equation*}

\noindent The control follows easily thanks to and Lemma \ref{boundary_control}, Theorem \ref{control_U}, Theorem \ref{high_order_estimate_U} and Remark \ref{control_omega_bott}.
\end{proof}

\begin{remark}\label{transport_eq_vort_surf}
\noindent Notice that we can also take the trace at the surface of the vorticity equation and we obtain a transport equation for $\underline{\omega} \cdot N^{\mu}$,

\upshape
\begin{equation}
\partial_{t} \left( \underline{\omega} \cdot N^{\mu} \right) + \epsilon \nabla \cdot \left( \left[\underline{\omega} \cdot N^{\mu} + \frac{1}{\text{Ro}} \right] \underline{\text{V}} \right) = 0.
\end{equation}
\itshape

\end{remark}

\subsection{Quasilinearization and a priori estimates}

\noindent In this part, we quasilinearize the system \eqref{Castro_lannes_formulation}. We introduce the Rayleigh-Taylor coefficient 

\begin{equation}\label{rayleigh_taylor_coefficient}
\mathfrak{a}: = \mathfrak{a}[\epsilon \zeta, \beta b](\psi, \omega) = 1 + \epsilon \left( \partial_{t} + \epsilon \underline{\text{V}}[\epsilon \zeta, \beta b](\psi, \omega) \cdot \nabla \right) \underline{\text{w}}[\epsilon \zeta, \beta b](\psi, \omega).
\end{equation}

\noindent It is well-known that the positivity of this quantity is essential for the wellposedness of the water waves equations (see for instance Remark 4.17 in \cite{Lannes_ww} or \cite{ebin_illposedness}). Thanks to Equation \eqref{eq_int1}, we can easily adapt Part 4.3.5 in \cite{Lannes_ww} and check that the positivity of $\mathfrak{a}$ is equivalent to the classical Rayleigh-Taylor criterion (\cite{Rayleigh-Taylor})

\begin{equation*}
\underset{\RD}{\inf} \left(- \partial_{z} \mathcal{P}_{|z=\epsilon \zeta} \right) > 0,
\end{equation*}

\noindent where we recall that $\mathcal{P}$ is the pressure in the fluid domain. We can now give a quasilinearization of \eqref{Castro_lannes_formulation_straight}. We recall that the notation $\underline{\text{U}}^{\mu}_{(\alpha)}$ is defined in \eqref{alinhac_U} and $\psia$ is defined in \eqref{alinhac_psi}.

\begin{prop}\label{quasilinearization}
\noindent Let \upshape$N \geq 5$, $T > 0$, $b \in L^{\infty} \cap \Hdot^{N+1}(\RD)$, $P \in L^{\infty}_{t} \left( \R^{+} ; \Hdot^{N+1}_{\! X} \left(\RD \right) \right)$ \itshape and \upshape $\left( \zeta, \psi, \omega \right) \in E^{N}_{T}$ \itshape solution of the system \eqref{Castro_lannes_formulation_straight} such that \upshape$\left(\zeta, b \right)$ \itshape satisfy Condition \eqref{nonvanishing} on \upshape$[0,T]$\itshape. Then, for \upshape$\alpha, \gamma \in \mathbb{N}^{d}$ \itshape and for \upshape$k \in \left\lbrace x,y \right\rbrace$ \itshape such that \itshape$ \partial^{\alpha} = \partial_{k} \partial^{\gamma}$ \upshape and \upshape$|\gamma| \leq N-1$\itshape, we have the following quasilinearization

\upshape
\begin{equation}\label{quasilinearized_system}
\begin{aligned}
&\left( \partial_{t} + \epsilon \underline{\text{V}} \cdot \nabla \right) \partial^{\alpha} \zeta - \frac{1}{\mu} \partial_{k} \underline{\text{U}}^{\mu}_{(\gamma)} \cdot N^{\mu} = R^{1}_{\alpha},\\
&\left( \partial_{t} + \epsilon \underline{\text{V}} \cdot \nabla \right) \left(\underline{\text{U}}^{\mu}_{(\gamma) \sslash} \cdot e_{\textbf{k}} \right) + \mathfrak{a} \partial^{\alpha} \zeta = - \partial^{\alpha} P + R^{2}_{\alpha},
\end{aligned}
\end{equation}
\itshape

\noindent where 

\upshape
\begin{equation}\label{control_remainder}
\left\lvert R^{1}_{\alpha} \right\rvert_{2} \hspace{-0.05cm} + \hspace{-0.05cm} \left\lvert R^{2}_{\alpha} \right\rvert_{2} \hspace{-0.05cm} + \hspace{-0.05cm} \left\lvert \mathfrak{P} R^{2}_{\alpha} \right\rvert_{2} \hspace{-0.05cm} \leq M_{N} \left(\max \left(\epsilon, \frac{\epsilon}{\text{Ro}} \right) \mathcal{E}^{N} \hspace{-0.1cm} \left( \zeta, \psi, \omega \right) + \frac{\epsilon}{\text{Ro}} \sqrt{\mathcal{E}^{N} \hspace{-0.1cm} \left( \zeta, \psi, \omega \right)} \right).
\end{equation}
\itshape

\end{prop}

\medskip

\noindent Before proving this result, we introduce the following notation. For $\alpha \in \mathbb{N}^{d}$ and $f,g \in H^{|\alpha|-1}(\RD)$, we define the symmetric commutator

\begin{equation*}
\left[ \partial^{\alpha }, f, g \right] = \partial^{\alpha} \left( fg \right) - g \partial^{\alpha} f - f \partial^{\alpha} g. 
\end{equation*}

\begin{proof}
\noindent Firstly, we apply $\partial^{\alpha}$ to the first equation of \eqref{Castro_lannes_formulation_straight}

\begin{equation*}
\partial_{t} \partial^{\alpha} \zeta + \epsilon \underline{\text{V}} \cdot \nabla \partial^{\alpha} \zeta + \epsilon \partial^{\alpha} \underline{\text{V}} \cdot \nabla \zeta - \frac{1}{\mu} \partial^{\alpha} \underline{\text{w}} + \epsilon \left[ \partial^{\alpha}, \underline{\text{V}}, \nabla \zeta \right] = 0.
\end{equation*}

\noindent Using Theorem \ref{high_order_estimate_U} and the trace Lemma \ref{trace_lemma}, we get the first equality. For the second equality we get, after applying $\partial_{k}$ to the second equation of \eqref{Castro_lannes_formulation},

\begin{align*}
\partial_{t} \partial_{k} \psi \hspace{-0.05cm} + \hspace{-0.05cm} \partial_{k} \zeta \hspace{-0.05cm} + \hspace{-0.05cm} \epsilon \underline{\text{V}} \cdot \hspace{-0.05cm} \left( \hspace{-0.1cm} \left( \partial_{k} \nabla \psi \hspace{-0.1cm} - \hspace{-0.1cm} \epsilon \underline{\text{w}} \nabla \partial_{k} \zeta \right) \hspace{-0.1cm} + \hspace{-0.1cm} \partial_{k} \nabla^{\perp} \psit \right) \hspace{-0.1cm} &- \frac{\epsilon}{\mu} \underline{\text{w}} \partial_{k} \left( \underline{\text{U}}^{\mu} \hspace{-0.1cm} \cdot \hspace{-0.1cm} N^{\mu} \right)\\
&- \epsilon \partial_{k} \frac{\nabla^{\perp}}{\Delta} \cdot \left(\hspace{-0.1cm} \left( \underline{\omega} \hspace{-0.05cm} \cdot \hspace{-0.05cm} N^{\mu} + \frac{1}{\text{Ro}} \right) \underline{\text{V}} \right) = - \partial_{k} P.
\end{align*}

\noindent Then, applying $\partial^{\gamma}$ and using Lemma 4.3 in \cite{Castro_Lannes_vorticity} (we can easily adapt it thanks to Theorem \ref{high_order_estimate_U} and Lemma \ref{psit_control}) we get 

\begin{align*}
\partial_{t} \partial^{\alpha} \psi + \partial^{\alpha} \zeta &+ \epsilon \underline{\text{V}} \cdot \left( \hspace{-0.05cm} \left( \partial^{\alpha} \nabla \psi \hspace{-0.1cm} -\hspace{-0.1cm} \epsilon \underline{\text{w}} \nabla \partial^{\alpha} \zeta \right) \hspace{-0.1cm} + \hspace{-0.1cm} \partial^{\alpha} \nabla^{\perp} \psit \right)\\
&- \frac{\epsilon}{\mu} \underline{\text{w}} \partial^{\alpha} \left( \underline{\text{U}}^{\mu} \hspace{-0.1cm} \cdot \hspace{-0.1cm} N^{\mu} \right) - \epsilon \partial^{\alpha} \frac{\nabla^{\perp}}{\Delta} \cdot \left(\hspace{-0.1cm} \left( \underline{\omega} \hspace{-0.05cm} \cdot \hspace{-0.05cm} N^{\mu} + \frac{1}{\text{Ro}} \right) \underline{\text{V}} \right) = - \partial^{\alpha} P + \widetilde{R^{2}_{\alpha}},
\end{align*}

\noindent where $\widetilde{R^{2}_{\alpha}}$ is controlled

\begin{equation}\label{control_remainder2}
\left\lvert \widetilde{R^{2}_{\alpha}} \right\rvert_{2} + \left\lvert \mathfrak{P} \widetilde{R^{2}_{\alpha}} \right\rvert_{2} \leq \epsilon M_{N} \mathcal{E}^{N} \hspace{-0.1cm} \left( \zeta, \psi, \omega \right).
\end{equation}

\noindent Using the first equation of \eqref{Castro_lannes_formulation} and the fact that $\Delta \psit = \underline{\omega} \cdot N^{\mu}$, we obtain

\begin{align*}
\partial_{t} \psia \hspace{-0.1cm} + \hspace{-0.1cm} \mathfrak{a} \partial^{\alpha} \zeta \hspace{-0.1cm} + \hspace{-0.1cm} \epsilon \underline{\text{V}} \cdot \nabla \psia \hspace{-0.1cm} + \hspace{-0.1cm} \frac{\epsilon}{\text{Ro}}  \partial^{\alpha} \frac{\nabla^{\perp}}{\Delta} \cdot \underline{\text{V}} \hspace{-0.1cm} + \hspace{-0.1cm} \partial^{\alpha} P &= \epsilon \partial^{\alpha} \frac{\nabla^{\perp}}{\Delta} \cdot \left( \underline{\omega} \cdot N^{\mu} \underline{\text{V}} \right)\\
&\hspace{2cm} -  \epsilon \underline{\text{V}} \cdot \nabla^{\perp} \partial^{\alpha} \psit +  \widetilde{R^{2}_{\alpha}}\\
&= \epsilon \hspace{-0.15cm} \underset{k \in \{1,2 \}}{\displaystyle{\sum}} \hspace{-0.15cm} (-1)^{k+1} \hspace{-0.13cm} \left[\partial^{\alpha} \frac{\partial_{k}}{\Delta}, \underline{\text{V}}_{\, 3-k} \right] \left( \underline{\omega} \cdot N^{\mu} \right)\\
&\hspace{0.5cm} + \widetilde{R^{2}_{\alpha}}\\
&:=\widetilde{R^{3}_{\alpha}} + \widetilde{R^{2}_{\alpha}},\\
\end{align*}

\noindent where $\partial_{1} = \partial_{x}$ and $\partial_{2} = \partial_{y}$. Then, using Theorem 3 in \cite{Lannes_sharp_estimates}, Lemma \ref{P_product} and Lemma \ref{boundary_control} we get

\begin{equation*}
\left\lvert \widetilde{R^{3}_{\alpha}} \right\rvert_{2} + \left\lvert \mathfrak{P} \widetilde{R^{3}_{\alpha}} \right\rvert_{2} \leq \epsilon M_{N} \llver \text{V} \rrver_{H^{N,1}} \llver \omega \rrver_{H^{N-1,1}} + \epsilon \left\lvert \mathfrak{P} \frac{\nabla^{\perp}}{\Delta} \cdot \left( \underline{\omega} \cdot N^{\mu} \partial^{\alpha} \underline{\text{V}} \right) \right\rvert_{2}.
\end{equation*}

\noindent Furthermore,

\begin{align*}
\left\lvert \mathfrak{P} \frac{\nabla^{\perp}}{\Delta} \cdot \left( \underline{\omega} \cdot N^{\mu} \partial^{\alpha} \underline{\text{V}} \right) \right\rvert_{2} \hspace{-0.1cm} &\hspace{-0.1cm}\leq \left\lvert \frac{1}{\sqrt{1+\sqrt{\mu}|D|}} \left( \underline{\omega} \cdot N^{\mu} \partial^{\alpha} \underline{\text{V}} \right) \right\rvert_{2},\\
&\hspace{-0.1cm} \leq \left\lvert \frac{1}{\sqrt{1+\sqrt{\mu}|D|}} \left( \partial_{k} \left(\underline{\omega} \cdot N^{\mu} \right) \partial^{\gamma} \underline{\text{V}} \right) \right\rvert_{2} \hspace{-0.2cm} + \hspace{-0.05cm} \left\lvert \mathfrak{P} \left( \underline{\omega} \cdot N^{\mu} \partial^{\gamma} \underline{\text{V}} \right) \right\rvert_{2} ,\\
&\leq C \left(\epsilon \left\lvert \zeta \right\rvert_{H^{N}} \right) \left\lvert \underline{\omega} \right\rvert_{H^{N-2}} \left( \left\lvert \underline{\text{V}} \right\rvert_{H^{N-1}} + \left\lvert \mathfrak{P}  \partial^{\gamma} \underline{\text{V}} \right\rvert_{2} \right),
\end{align*}

\noindent where we use Lemma \ref{P_commutator}. The first term is controlled thanks to the trace Lemma \ref{trace_lemma} and Theorem \ref{high_order_estimate_U}. For the second term, we have

\begin{equation*}
\partial^{\gamma} \underline{\text{V}} = \nabla \partial^{\gamma} \psi - \epsilon \underline{\text{w}} \nabla \partial^{\gamma} \zeta - \epsilon \partial^{\gamma} \underline{\text{w}} \nabla \zeta + \nabla^{\perp} \partial^{\gamma} \psit - \epsilon \left[\partial^{\gamma},  \underline{\text{w}} , \nabla \zeta \right],
\end{equation*}

\noindent and the control follows from  Lemma \ref{P_product}, Lemma \ref{boundary_control}, Theorem \ref{high_order_estimate_U} and Lemma \ref{psit_control}. Then, we obtain

\begin{equation*}
\partial_{t} \psia + \mathfrak{a} \partial^{\alpha} \zeta + \epsilon \underline{\text{V}} \cdot \nabla \psia +  \frac{\epsilon}{\text{Ro}}  \partial^{\alpha} \frac{\nabla^{\perp}}{\Delta} \cdot \underline{\text{V}} + \partial^{\alpha} P = \widetilde{\widetilde{R^{2}_{\alpha}}},
\end{equation*}

\noindent where $\widetilde{\widetilde{R^{2}_{\alpha}}}$ satisfied also the estimate \eqref{control_remainder2}. Finally, we can adapt Lemma 4.4 in \cite{Castro_Lannes_vorticity} thanks to Remark \ref{transport_eq_vort_surf}, Theorem \ref{high_order_estimate_U} and Proposition  \ref{time_derivative_control} and we get

\begin{equation*}
\partial_{t} \psia = \partial_{t} \left( \text{U}^{\mu}_{(\gamma) \sslash} \cdot \textbf{e}_{k} \right) + \widetilde{R_{\alpha}},
\end{equation*}

\noindent where $\widetilde{R_{\alpha}}$ satisfies the same estimate as $R_{2}$ in \eqref{control_remainder}. The third equality is a direct consequence of Proposition \ref{vorticity_estimate}.
\end{proof}

\noindent In order to establish an a priori estimate we need to control the Rayleigh-Taylor coefficient $\mathfrak{a}$. The following Proposition is adapted from Proposition 2.10 in \cite{my_proud_res}.

\begin{prop}\label{controls_rt}
Let $T > 0$, $t_{0}> \frac{d}{2}$, $N \geq 5$, $(\zeta,\psi,\omega) \in E_{T}^{N}$ is a solution of the water waves equations \eqref{Castro_lannes_formulation_straight}, $P \in L^{\infty}(\R^{+}; \Hdot^{N+1}(\RD))$ and $b \in L^{\infty} \cap \Hdot^{N+1}(\RD)$, such that Condition \eqref{nonvanishing} is satisfied. We assume also that $\epsilon, \beta, \text{Ro}, \mu$ satisfy \eqref{constraints_parameters}. Then, for all $0 \leq t \leq T$,

\begin{equation*}
\lver \mathfrak{a} -1 \rver_{W^{1,\infty}} \leq C\left( M_{N}, \epsilon \sqrt{\mathcal{E}^{N} \hspace{-0.1cm} \left( \zeta, \psi, \omega \right)} \right) \epsilon \sqrt{\mathcal{E}^{N} \hspace{-0.1cm} \left( \zeta, \psi, \omega \right)} + \epsilon M_{N} \left\lvert \nabla P \right\rvert_{L^{\infty}_{t} H_{\! X}^{N}}.
\end{equation*}

\noindent Furthermore, if $\partial_{t} P \in L^{\infty}(\R^{+}; \Hdot^{N}(\RD))$, then,

\begin{equation*}
\lver \partial_{t} \mathfrak{a} \rver_{L^{\infty}} \leq C \hspace{-0.1cm} \left( \hspace{-0.1cm} M_{N}, \left\lvert \nabla P \right\rvert_{L^{\infty}_{t} H_{\! X}^{N}}, \epsilon \sqrt{\mathcal{E}^{N} \hspace{-0.1cm} \left( \zeta, \psi, \omega \right)} \right) \epsilon \sqrt{\mathcal{E}^{N} \hspace{-0.1cm} \left( \zeta, \psi, \omega \right)} + \epsilon M_{N} \left\lvert \nabla P \right\rvert_{W^{1,\infty}_{t} H_{\! X}^{N}}.
\end{equation*}
\end{prop}

\begin{proof}
\noindent Using Proposition \ref{time_derivative_control} we get that

\begin{equation}\label{rt_other_expression}
\begin{aligned}
\mathfrak{a}[\epsilon \zeta, \beta b](\psi, \omega) &= 1 + \epsilon^{2} \underline{\text{V}} \cdot \nabla \underline{\text{w}} + \epsilon \partial_{t} \zeta \underline{\partial_{z}^{\sigma} \text{w}}\\
&\hspace{0.1cm} + \epsilon \underline{\text{w}}[\epsilon \zeta, \beta b] \left(\partial_{t} \psi - \epsilon \underline{\text{w}}[\epsilon \zeta, \beta b](\psi,\omega) \partial_{t} \zeta + \epsilon \sqrt{\mu} \frac{\nabla}{\Delta} \cdot \left(\underline{\omega_{h}}^{\perp} \partial_{t} \zeta \right), \partial_{t}^{\sigma} \omega \right).\\
\end{aligned}
\end{equation}

\noindent Then, using the equations satisfied by $\left(\zeta,\psi,\omega\right)$, Theorems \ref{control_U} and \ref{high_order_estimate_U}, Remark \ref{control_omega_bott} and standard controls, we easily get the first inequality. The second inequality can be proved similarly.
\end{proof}

\noindent We can now establish an a priori estimate for the Castro-Lannes System with a Coriolis forcing under the positivity on the Rayleigh-Taylor coefficient

\begin{equation}\label{rayleigh_taylor_assumption}
  \exists \, \mathfrak{a}_{\min} > 0 \text{ ,  } \mathfrak{a} \geq \mathfrak{a}_{\min}.
\end{equation}

\begin{thm}\label{energy_estimate}
\noindent Let \upshape$N \geq 5$, $T > 0$, $b \in L^{\infty} \cap \Hdot^{N+2}(\RD)$, $P \in L^{\infty}_{t} \left( \R^{+}; \Hdot^{N+1}(\RD) \right)$ \itshape and \upshape $\left( \zeta, \psi, \omega \right) \in E^{N}_{T}$ \itshape solution of the water waves equations \eqref{Castro_lannes_formulation_straight} such that \upshape$\left(\zeta, b \right)$ \itshape satisfy Condition \eqref{nonvanishing} and \upshape $\mathfrak{a}[\epsilon \zeta, \beta b] \left( \psi, \omega \right)$ \itshape satisfies \eqref{rayleigh_taylor_assumption} on \upshape$\left[0,T \right]$\itshape. We assume also that $\epsilon, \beta, \text{Ro}, \mu$ satisfy \eqref{constraints_parameters}. Then, for all \upshape $t \in \left[0,T \right]$, \itshape

\upshape
\small
\begin{equation}
\begin{aligned}
\frac{d}{dt} \mathcal{E}^{N} \hspace{-0.1cm} \left( \zeta, \psi, \omega \right) &\leq \hspace{-0.05cm} C \hspace{-0.05cm} \left(\hspace{-0.05cm} \mu_{\max}, \hspace{-0.05cm} \frac{1}{h_{\min}}, \hspace{-0.05cm} \epsilon \sqrt{\mathcal{E}^{N} \hspace{-0.1cm} \left( \zeta, \psi, \omega \right)}, \hspace{-0.05cm} \beta \left\lvert \nabla b \right\rvert_{H^{N+1}}, \hspace{-0.05cm} \beta \left\lvert b \right\rvert_{L^{\infty}}, \lver \nabla P \rver_{W^{1,\infty}_{t} H^{N}_{\! X}} \right) \times  \\
&\hspace{1cm} \left( \epsilon \mathcal{E}^{N} \hspace{-0.1cm} \left( \zeta, \psi, \omega \right)^{\frac{3}{2}} + \max \left(\epsilon, \beta, \frac{\epsilon}{\text{Ro}} \right) \mathcal{E}^{N} \hspace{-0.1cm} \left( \zeta, \psi, \omega \right) + \lver \nabla P \rver_{L^{\infty}_{t} H^{N}_{\! X}} \sqrt{\mathcal{E}^{N} \hspace{-0.1cm} \left( \zeta, \psi, \omega \right)} \right).
\end{aligned}
\end{equation}
\normalsize
\itshape

\end{thm}

\begin{proof}
\noindent Compared to \cite{Castro_Lannes_vorticity}, we have here a non flat bottom, a Coriolis forcing and a non constant pressure. We focus on these terms. Inspired by \cite{Castro_Lannes_vorticity} we can symmetrize the Castro-Lannes system. We define a modified energy

\begin{equation}\label{modif_energy}
\begin{aligned}
\mathcal{F}^{N} \left( \psi, \zeta, \omega \right) = \frac{1}{2} \Big( \llver \bm{\omega} \rrver_{H^{N-1}}^{2} &+  \left\lvert \frac{1}{\mathfrak{P}} \left( \omega_{b} \cdot N_{b}^{\mu} \right) \right\rvert_{2}^{2} +  \underset{|\alpha| \leq 3}{\sum} \left\lvert \partial^{\alpha} \zeta \right\rvert_{2}^{2} + \frac{1}{\mu} \int_{\mathcal{S}} \left(1+\partial_{z} \sigma \right) \left\lvert \partial^{\alpha} \text{U}^{\mu} \right\rvert^{2} \\
&\hspace{1cm} + \hspace{-1cm} \underset{k=x,y, 1 \leq |\gamma| \leq N-1}{\sum} \hspace{-1cm} \left( \mathfrak{a} \partial_{k} \partial^{\gamma} \zeta,  \partial_{k} \partial^{\gamma} \zeta \right) + \frac{1}{\mu} \int_{\mathcal{S}} \left(1+\partial_{z} \sigma \right) \left\lvert \partial_{k} \text{U}^{\mu}_{(\gamma)} \right\rvert^{2}  \Big).
\end{aligned}
\end{equation}

\noindent From Proposition \ref{div_curl_good_unknowns_control} and Proposition \ref{frakP_psi_H3} we get 

\begin{equation*}
\mathcal{E}^{N} \hspace{-0.1cm} \left( \psi, \zeta, \omega \right) \leq C \left( \frac{1}{\mathfrak{a}_{\min}}, M_{N} \right) \mathcal{F}^{N} \hspace{-0.1cm} \left( \psi, \zeta, \omega \right),
\end{equation*}

\noindent and from Theorem \ref{control_U}, Theorem \ref{high_order_estimate_U}, Remark \ref{control_omega_bott} and Proposition \ref{controls_rt} we obtain that

\small
\begin{equation*}
\mathcal{F}^{N} \hspace{-0.1cm} \left( \psi, \zeta, \omega \right) \leq C \left(\frac{1}{h_{\min}}, \beta \left\lvert b \right\rvert_{L^{\infty}}, \beta \left\lvert \nabla b \right\rvert_{H^{N}}, \left\lvert \nabla P \right\rvert_{L^{\infty}_{t} H^{N}_{\! X}}, \epsilon \sqrt{\mathcal{E}^{N} \hspace{-0.1cm} \left( \psi, \zeta, \omega \right) } \right) \mathcal{E}^{N} \hspace{-0.1cm} \left( \psi, \zeta, \omega \right).
\end{equation*}
\normalsize

\noindent Hence, in the following we estimate $\frac{d}{dt} \mathcal{F}^{N} \left( \psi, \zeta, \omega \right)$. We already did the work for the vorticity in Proposition \ref{vorticity_estimate}. In the following $R$ will be a remainder whose exact value has no importance and satisfying

\begin{equation}\label{control_remainder3}
\lver R \rver_{2} \leq  C \hspace{-0.1cm} \left(\frac{1}{h_{\min}}, \beta \left\lvert b \right\rvert_{L^{\infty}}, \beta \left\lvert \nabla b \right\rvert_{H^{N+1}}, \left\lvert \nabla P \right\rvert_{W^{1,\infty}_{t} H^{N}_{\! X}}, \epsilon \sqrt{\mathcal{E}^{N} \hspace{-0.1cm} \left( \psi, \zeta, \omega \right) } \right) \mathcal{E}^{N} \hspace{-0.1cm} \left( \psi, \zeta, \omega \right) \hspace{-0.1cm}.
\end{equation}

\noindent We start by the low order terms. Let $\alpha \in \mathbb{N}^{d}$, $|\alpha| \leq 3$. We apply $\partial^{\alpha}$ to the first equation of System \eqref{Castro_lannes_formulation_straight} and we multiply it by $\zeta$. Then, we apply  $\partial^{\alpha}$ to the second equation and we multiply it by $\frac{1}{\mu} \underline{\text{U}}^{\mu} \cdot N^{\mu}$. By summing these two equations, we obtain, thanks to Theorem \ref{control_U}, Theorem \ref{high_order_estimate_U}, Remark \ref{control_omega_bott} and the trace Lemma,

\begin{equation}\label{inequality1}
\begin{aligned}
\frac{1}{2} \partial_{t} \left(\partial^{\alpha} \zeta, \partial^{\alpha} \zeta \right) + \left( \partial_{t} \partial^{\alpha} \psi, \frac{1}{\mu} \partial^{\alpha} \underline{\text{U}}^{\mu} \cdot N^{\mu} \right) \hspace{-0.1cm} &+ \hspace{-0.05cm} \frac{\epsilon}{\text{Ro}} \left( \frac{\nabla}{\Delta} \cdot \partial^{\alpha} \underline{\text{V}}^{\perp} , \frac{1}{\mu} \partial^{\alpha} \underline{\text{U}}^{\mu} \cdot N^{\mu} \right)\\
&\hspace{0.5cm} + \left(\partial^{\alpha} P, \frac{1}{\mu} \partial^{\alpha} \underline{\text{U}}^{\mu} \cdot N^{\mu} \right) \leq \epsilon \lver R \rver_{2}.
\end{aligned}
\end{equation}

\noindent Furthermore, using again the same Propositions as before, we get

\small
\begin{equation*}
\frac{\epsilon}{\text{Ro}} \hspace{-0.05cm} \left(\hspace{-0.05cm} \frac{\nabla}{\Delta} \hspace{-0.05cm} \cdot \hspace{-0.05cm} \partial^{\alpha} \underline{\text{V}}^{\perp} \hspace{-0.05cm}, \frac{1}{\mu} \hspace{-0.05cm} \partial^{\alpha}  \underline{\text{U}}^{\mu} \hspace{-0.05cm} \cdot \hspace{-0.05cm} N^{\mu} \hspace{-0.05cm} \right) \hspace{-0.08cm}+\hspace{-0.08cm} \left(\hspace{-0.05cm} \partial^{\alpha} \hspace{-0.05cm} P, \frac{1}{\mu} \hspace{-0.05cm} \partial^{\alpha} \underline{\text{U}}^{\mu} \hspace{-0.05cm} \cdot \hspace{-0.05cm} N^{\mu} \right) \hspace{-0.05cm} \leq \frac{\epsilon}{\text{Ro}} \lver R \rver_{2} + M_{N} \hspace{-0.05cm} \left\lvert \nabla P \right\rvert_{L^{\infty}_{t} H^{N}_{\! X}} \hspace{-0.1cm} \sqrt{\hspace{-0.05cm} \mathcal{E}^{N} \hspace{-0.05cm} \left( \psi, \zeta, \omega \right)}.
\end{equation*}
\normalsize

\noindent Then, we have to link $\left( \partial_{t} \partial^{\alpha} \psi, \partial^{\alpha}  \underline{\text{U}}^{\mu} \cdot N^{\mu} \right)$ to $\partial_{t} \int_{\mathcal{S}} \left(1+\partial_{z} \sigma \right) \left\lvert \partial^{\alpha} \text{U}^{\mu} \right\rvert^{2}$. Remarking that $\psi =\underline{\phi}$, where $\phi$ satisfies

\begin{equation}\label{lap_problem}
\left\{
\begin{aligned}
&\nabla^{\mu}_{\! X,z} \cdot P(\Sigma) \nabla^{\mu}_{\! X,z} \phi = 0 \text{ in } \mathcal{S},\\
&\phi_{|z=0} = \psi \text{, } \text{e}_{z} \hspace{-0.05cm} \cdot P(\Sigma) \hspace{-0.05cm} \nabla^{\mu}  \phi_{|z=-1} = 0,
\end{aligned}
\right.
\end{equation}

\noindent we get thanks to Green's identity

\begin{equation*}
\begin{aligned}
\left(\partial_{t} \partial^{\alpha} \psi, \frac{1}{\mu} \partial^{\alpha} \underline{\text{U}}^{\mu} \cdot N^{\mu} \right) &= \frac{1}{\mu} \int_{\mathcal{S}} \hspace{-0.05cm} \left(1+\partial_{z} \sigma \right) \nabla^{\sigma, \mu}_{\! X,z} \left( \partial_{t}   \partial^{\alpha} \phi \right) \cdot \partial^{\alpha} \text{U}^{\mu}\\
& + \frac{1}{\mu} \int_{\mathcal{S}} \hspace{-0.05cm} \left(1+\partial_{z} \sigma \right) \partial^{\alpha} \partial_{t} \phi  \nabla^{\sigma, \mu}_{\! X,z} \cdot \partial^{\alpha} \text{U}^{\mu} + \left(\partial_{t} \partial^{\alpha} \phi_{b}, \frac{1}{\mu} \partial^{\alpha} \text{U}^{\mu}_{b} \cdot N^{\mu}_{b} \right) \hspace{-0.1cm}.
\end{aligned}
\end{equation*}

\noindent Then, notice that $\partial_{k} = \partial_{k}^{\sigma} + \partial_{k} \sigma \partial_{z}^{\sigma}$ for  $k \in \left\{t, x,y \right\}$ and $ \partial_{k}^{\sigma}$ and $\nabla^{\sigma, \mu}_{\! X,z}$ commute. We differentiate Equation \eqref{lap_problem} with respect to $t$ and we obtain thanks to Theorems \ref{control_U}, \ref{high_order_estimate_U}, Proposition \ref{time_derivative_control} and Lemma 2.38 in \cite{Lannes_ww} (irrotational theory),

\begin{equation*}
\begin{aligned}
\left(\hspace{-0.1cm} \partial_{t} \partial^{\alpha} \psi, \frac{1}{\mu} \partial^{\alpha} \underline{\text{U}}^{\mu} \cdot N^{\mu} \hspace{-0.1cm} \right) &= \frac{1}{\mu} \int_{\mathcal{S}} \hspace{-0.1cm} \left(1+\partial_{z} \sigma \right) \partial_{t}^{\sigma} \partial^{\sigma, \alpha} \nabla^{\sigma, \mu}_{\! X,z} \phi  \cdot \partial^{\alpha} \text{U}^{\mu}\\
 &\hspace{1cm} + \left(\partial_{t} \partial^{\alpha} \phi_{b}, \frac{1}{\mu} \partial^{\alpha} \text{U}^{\mu}_{b} \cdot N^{\mu}_{b} \right) + \max(\epsilon,\beta) R.
 \end{aligned}
\end{equation*}

\noindent Using the fact that $\text{w}_{b} = \mu \beta \nabla b \cdot \text{V}_{b}$, we get

\begin{equation*}
\left(\hspace{-0.1cm} \partial_{t} \partial^{\alpha} \phi_{b}, \frac{1}{\mu} \partial^{\alpha} \text{U}^{\mu}_{b} \cdot N^{\mu}_{b} \hspace{-0.1cm} \right) \leq \beta M_{N} \lver \partial_{t} \partial^{\alpha} \phi_{b} \rver \sqrt{\hspace{-0.05cm} \mathcal{E}^{N} \hspace{-0.05cm} \left( \psi, \zeta, \omega \right)}.
\end{equation*}

\noindent Then, by the trace Lemma, we finally obtain 
 
\begin{equation*}
\left(\hspace{-0.1cm} \partial_{t} \partial^{\alpha} \phi_{b}, \frac{1}{\mu} \partial^{\alpha} \text{U}^{\mu}_{b} \cdot N^{\mu}_{b} \hspace{-0.1cm} \right) \leq \beta \lver R \rver_{2}.
\end{equation*}

\noindent Furthermore, remarking that $\text{U}^{\mu} = \nabla^{\sigma, \mu}_{\! X,z} \phi + \text{U}^{\sigma, \mu} [\epsilon \zeta, \beta b] \left( 0, \omega \right)$, we obtain, thanks to Proposition \ref{time_derivative_control}, Theorem \ref{control_U} and Theorem \ref{high_order_estimate_U},

\begin{equation*}
\left(\partial_{t} \partial^{\alpha} \psi, \frac{1}{\mu} \partial^{\alpha} \underline{\text{U}}^{\mu} \cdot N^{\mu} \right) = \frac{1}{\mu} \int_{\mathcal{S}} \left(1+\partial_{z} \sigma \right) \partial_{t} \partial^{\alpha} \text{U}^{\mu} \cdot \partial^{\alpha} \text{U}^{\mu} + \max \left( \epsilon, \beta, \frac{\epsilon}{\text{Ro}} \right) R.
\end{equation*}

\noindent Using the following identity

\begin{equation}\label{time_integral}
\partial_{t} \int_{\mathcal{S}} (1+\partial_{z} \sigma)fg = \int_{\mathcal{S}} (1+\partial_{z} \sigma) \partial_{t}^{\sigma}f g + \int_{\mathcal{S}} (1+\partial_{z} \sigma) f \partial_{t}^{\sigma} g  + \int_{\RD} \epsilon \partial_{t} \zeta \underline{f} \underline{g},
\end{equation}

\noindent we obtain that

\begin{equation*}
\frac{1}{\mu} \partial_{t} \int_{\mathcal{S}}  \hspace{-0.15cm} \left(1+\partial_{z} \sigma \right)  \hspace{-0.05cm} \left\lvert \partial^{\alpha} \text{U}^{\mu} \right\rvert^{2} \hspace{-0.05cm} \leq \max \left( \epsilon, \beta, \frac{\epsilon}{\text{Ro}} \right) \lvert R \rvert_{2} + M_{N} \hspace{-0.1cm} \left\lvert \nabla P \right\rvert_{L^{\infty}_{t} H^{N}_{\! X}}  \hspace{-0.1cm} \sqrt{\mathcal{E}^{N} \left( \psi, \zeta, \omega \right)}.
\end{equation*}

\noindent To control the high order terms of $\mathcal{F}^{N} \left( \psi, \zeta, \omega \right)$ we adapt Step 2 in Proposition 4.5 in \cite{Lannes_ww}. Thanks to Proposition \ref{quasilinearization}, we have

\begin{equation*}
\begin{aligned}
&\left( \partial_{t} + \epsilon \underline{\text{V}} \cdot \nabla \right) \partial^{\alpha} \zeta - \frac{1}{\mu} \partial_{k} \underline{\text{U}}^{\mu}_{(\gamma)} \cdot N^{\mu} = R^{1}_{\alpha},\\
&\left( \partial_{t} + \epsilon \underline{\text{V}} \cdot \nabla \right) \left(\underline{\text{U}}^{\mu}_{(\gamma) \sslash} \cdot e_{\textbf{k}} \right) + \mathfrak{a} \partial^{\alpha} \zeta = - \partial^{\alpha} P + R^{2}_{\alpha}.\\
\end{aligned}
\end{equation*}

\noindent Then, we multiply the first equation by $\mathfrak{a} \partial^{\alpha} \zeta$ and the second by $\frac{1}{\mu} \partial_{k} \underline{\text{U}}^{\mu}_{(\gamma)} \cdot N^{\mu}$ and we integrate over $\RD$. Then, using Propositions  \ref{control_U}, \ref{speed_surface_control} and \ref{controls_rt},

\small
\begin{equation*}
\begin{aligned}
\frac{1}{2} \partial_{t} \left(\mathfrak{a} \partial^{\alpha} \zeta, \partial \zeta \right) + \left(\left( \partial_{t} + \epsilon \underline{\text{V}} \cdot \nabla \right) \left(\underline{\text{U}}^{\mu}_{(\gamma) \sslash} \cdot e_{\textbf{k}} \right), \frac{1}{\mu} \partial_{k} \underline{\text{U}}^{\mu}_{(\gamma)} \cdot N^{\mu} \right) &\leq \epsilon \mathcal{E}^{N} \hspace{-0.1cm} \left( \psi, \zeta, \omega \right)^{\frac{3}{2}}\\
&\hspace{-3cm} + \max \left(\epsilon, \frac{\epsilon}{\text{Ro}} \right) \lver R \rver_{2} + M_{N} \hspace{-0.05cm} \left\lvert \nabla P \right\rvert_{L^{\infty}_{t} H^{N}_{\! X}} \hspace{-0.1cm} \sqrt{\hspace{-0.05cm} \mathcal{E}^{N} \hspace{-0.1cm} \left( \psi, \zeta, \omega \right)}.
\end{aligned}
\end{equation*}
\normalsize

\noindent Then, we remark that 

\begin{equation*}
\left( \partial_{t} + \epsilon \underline{\text{V}} \cdot \nabla \right) \left(\underline{\text{U}}^{\mu}_{(\gamma) \sslash} \cdot e_{\textbf{k}} \right) = \underline{\left( \partial_{t}^{\sigma} + \frac{\epsilon}{\mu} \text{U}^{\mu} \cdot \nabla^{\sigma,\mu} \right) \left(\text{U}^{\mathfrak{b},\mu}_{(\gamma) \sslash} \cdot e_{\textbf{k}} \right)},
\end{equation*}

\noindent where $\text{U}^{\mathfrak{b},\mu}_{(\gamma) \sslash} = \text{V}_{(\gamma)} + \text{w}_{(\gamma)} \nabla \sigma$. Then, we have

\begin{equation*}
\begin{aligned}
\left(\left( \partial_{t} + \epsilon \underline{\text{V}} \cdot \nabla \right) \left(\underline{\text{U}}^{\mu}_{(\gamma) \sslash} \cdot e_{\textbf{k}} \right), \frac{1}{\mu} \partial_{k} \underline{\text{U}}^{\mu}_{(\gamma)} \cdot N^{\mu} \right)&\\
&\hspace{-4cm} = \frac{1}{\mu} \int_{\mathcal{S}} (1+\partial_{z} \sigma) \left( \partial_{t}^{\sigma} + \frac{\epsilon}{\mu} \text{U}^{\mu} \cdot \nabla^{\sigma,\mu} \right) \left(\text{U}^{\mathfrak{b},\mu}_{(\gamma) \sslash} \cdot e_{\textbf{k}} \right) \nabla^{\sigma,\mu}_{\! X,z} \cdot \left( \partial_{k} \text{U}^{\mu}_{(\gamma)} \right)\\
&\hspace{-4cm} + \frac{1}{\mu} \int_{\mathcal{S}} (1+\partial_{z} \sigma) \nabla^{\sigma,\mu}_{\! X,z} \left( \partial_{t}^{\sigma} + \frac{\epsilon}{\mu} \text{U}^{\mu} \cdot \nabla^{\sigma,\mu} \right) \left(\text{U}^{\mathfrak{b},\mu}_{(\gamma) \sslash} \cdot e_{\textbf{k}} \right) \left( \partial_{k} \text{U}^{\mu}_{(\gamma)} \right)\\
&\hspace{-4cm} + \left( \left( \partial_{t} + \epsilon \text{V}_{b} \cdot \nabla \right) \left(\text{U}^{\mathfrak{b},\mu}_{(\gamma) \sslash} \cdot e_{\textbf{k}} \right)_{b}, \frac{1}{\mu} \partial_{k} \left(\text{U}^{\mu}_{(\gamma)}\right)_{b} \cdot N^{\mu}_{b} \right).
\end{aligned}
\end{equation*}

\noindent We focus on the last term (bottom contribution). The two other terms can be controlled as in Step 2 in Proposition 4.5 in \cite{Castro_Lannes_vorticity}. Using the same computations as in Proposition \ref{speed_surface_control}, we have

\begin{equation*}
\frac{1}{\mu} \partial_{k} \left(\text{U}^{\mu}_{(\gamma)}\right)_{b} \cdot N^{\mu}_{b} = - \mu \beta \nabla \partial^{\alpha} b \cdot \text{V}_{b} + \text{l.o.t},
\end{equation*}

\noindent where l.o.t stands for lower order terms that can be controlled by the energy. Then, since $b \in \Hdot^{N+2}(\RD)$, we have by standard controls,

\begin{equation*}
\lver \frac{1}{\mu} \partial_{k} \left(\text{U}^{\mu}_{(\gamma)}\right)_{b} \cdot N^{\mu}_{b} \rver_{H^{\frac{1}{2}}} \leq \beta \lver \nabla b \rver_{H^{N+1}} \sqrt{\hspace{-0.05cm} \mathcal{E}^{N} \hspace{-0.05cm} \left( \psi, \zeta, \omega \right)}.
\end{equation*}

\noindent Furthermore, using Propositions \ref{control_U}, \ref{high_order_estimate_U} and \ref{time_derivative_control} and standard controls, we have

\begin{equation*}
\lver \left( \partial_{t} + \epsilon \text{V}_{b} \cdot \nabla \right) \left(\text{U}^{\mathfrak{b},\mu}_{(\gamma) \sslash} \cdot e_{\textbf{k}} \right)_{b} \rver_{H^{-\frac{1}{2}}} \leq \epsilon \lver R \rver_{2} + M_{N} \sqrt{\hspace{-0.05cm} \mathcal{E}^{N} \hspace{-0.05cm} \left( \psi, \zeta, \omega \right)},
\end{equation*}

\noindent and the control follows easily.
\end{proof}

\subsection{Existence result}\label{Existence_result}

\noindent We can now establish our existence theorem. Notice that thanks to Equation \eqref{rt_other_expression}, we can define the Rayleigh-Taylor coefficient at time $t=0$.

\begin{thm}\label{existence}
\noindent Let \upshape$A > 0$, $N \geq 5$, $b \in L^{\infty} \cap \Hdot^{N+2}	\! \left(\RD \right)$, $P \in W^{1,\infty}(\mathbb{R}^{+}; \Hdot^{N+1}(\RD))$, $\left( \zeta_{0}, \psi_{0}, \omega_{0} \right) \in \text{E}^{N}_{0}$\itshape such that \upshape$\nabla_{\! X,z}^{\sigma, \mu} \cdot \omega_{0} =0$\itshape. We suppose that \upshape$\left( \epsilon,\beta,\mu,\text{Ro} \right)$ \itshape satisfy \eqref{constraints_parameters}. We assume also that

\begin{equation*}
\exists \, h_{\min} \text{, } \mathfrak{a}_{\min} > 0 \text{ ,  } \epsilon \zeta _{0}+ 1 - \beta b \geq h_{\min} \text{ and } \mathfrak{a}[\epsilon \zeta, \beta b] \left(\psi, \omega \right)_{|t=0} \geq \mathfrak{a}_{\min}
\end{equation*}
\itshape

\noindent and 

\begin{equation*}
\mathcal{E}^{N} \hspace{-0.1cm} \left( \zeta_{0}, \psi_{0}, \omega_{0} \right) + \left\lvert \nabla P \right\rvert_{L^{\infty}_{t} H_{\! X}^{N}} \leq A.
\end{equation*}

 \itshape
 
\noindent Then, there exists \upshape$T > 0$, \itshape and a unique solution \upshape$\left( \zeta, \psi, \omega \right) \in \text{E}^{N}_{T}$ \itshape to the water waves equations \eqref{Castro_lannes_formulation_straight} with initial data \upshape$\left( \zeta_{0}, \psi_{0}, \omega_{0} \right)$\itshape. Moreover, 

\upshape
\begin{equation*}
T = \min \left( \frac{T_{0}}{\max(\epsilon, \beta, \frac{\epsilon}{\text{Ro}})}, \frac{T_{0}}{ \left\lvert \nabla P \right\rvert_{L^{\infty}_{t} H_{\! X}^{N}}} \right) \text{ , }  \frac{1}{T_{0}} =c^{1} \text{ and  } \underset{t \in [0,T]}{\sup} \mathcal{E}^{N} \left( \zeta(t), \psi(t), \omega(t) \right) = c^{2},
\end{equation*}
\itshape

\noindent with \upshape $c^{j} = C \left(A, \mu_{\max}, \frac{1}{h_{\min}}, \frac{1}{\mathfrak{a}_{\min}},\left\lvert b \right\rvert_{L^{\infty}}, \left\lvert \nabla b \right\rvert_{H^{N+1}}, \left\lvert \nabla P \right\rvert_{W^{1,\infty}_{t} H_{\! X}^{N}} \right)$.\itshape
\end{thm}

\begin{proof}
\noindent We do not give the proof. It is very similar to Theorem 4.7 in \cite{Castro_Lannes_vorticity}. We can regularize the system \eqref{Castro_lannes_formulation_straight} (see Step 2 of the proof of Theorem 4.7 in \cite{Castro_Lannes_vorticity}) and thanks to the energy estimate of Theorem \ref{energy_estimate} we get the existence. The uniqueness mainly follows from a similar proposition to Corollary 3.19 in \cite{Castro_Lannes_vorticity} which shows that the operator $\text{U}^{\sigma, \mu}$ has a Lipschitz dependence on its coefficients.
\end{proof}

\section{The nonlinear shallow water equations}

\subsection{The context}

\noindent In this part we justify rigorously the derivation of the nonlinear rotating shallow water equations from the water waves equations. We recall that, in this paper, we do not consider fast Coriolis forcing, i.e $\text{Ro} \leq \epsilon$. The nonlinear shallow water equations (or Saint Venant equations) is a model used by the mathematical and physical communities to study the water waves in shallow waters. Coupled with a Coriolis term, we usually describe shallow waters under the influence of the Coriolis force thanks to it  (see for instance \cite{Bresch_shallow_water}, \cite{Majda_ocean} or \cite{vilibic_Proudman_resonance}). But to our knowledge, there is no mathematical justification of this fact. Without the Coriolis term, many authors mathematically justify the Saint Venant equations; for the irrotationnal case, there are, for instance the works of Iguchi \cite{Iguchi_shallow_water} and Alvarez-Samaniego and Lannes (\cite{Alvarez_Lannes}). It is also done in \cite{Lannes_ww}. More recently, Castro and Lannes proposed a way to justify the Saint-Venant equations without the irrotational condition(\cite{Castro_Lannes_shallow_water} and \cite{Castro_Lannes_vorticity}), we address here the case in which the Coriolis force is present. We denote the depth

\begin{equation}
h(t,X) = 1 + \epsilon \zeta(t,X) - \beta b(X),
\end{equation}

\noindent and the averaged horizontal velocity

\begin{equation}
\overline{\textbf{V}} = \overline{\textbf{V}}[\epsilon \zeta, \beta b](\psi,\bm{\omega})(t,X) = \frac{1}{h(t,X)} \int_{z=-1+\beta b(X)}^{\epsilon \zeta(t,X)} \textbf{V}[\epsilon \zeta, \beta b](\psi,\bm{\omega})(t,X,z) dz.
\end{equation}

\noindent The Saint-Venant equations (in the nondimensionalized form) are

\begin{equation}\label{shallow_water_eq}
\left\{
\begin{array}{l}
\partial_{t} \zeta + \nabla \cdot (h \overline{\textbf{V}}) = 0,\\
\partial_{t} \overline{\textbf{V}} + \epsilon \left( \overline{\textbf{V}} \cdot \nabla \right) \overline{\textbf{V}} + \nabla \zeta + \frac{\epsilon}{\text{Ro}} \overline{\textbf{V}}^{\perp} = -\nabla P.
\end{array}
\right.
\end{equation}

\noindent It is well-known that the shallow water equations are wellposed (see Chapter 6 in \cite{Lannes_ww}  or \cite{Alvarez_Lannes} without the pressure term and the Coriolis forcing and \cite{Bresch_shallow_water}) and that we have the following Proposition.

\begin{prop}\label{local_existence_SW}
\noindent Let \upshape$t_{0} > \frac{d}{2}$, $s \geq t_{0}+1$ \itshape and \upshape$\zeta_{0}, b \in H^{s}(\RD)$, $\overline{V}_{0} \in H^{s}(\RD)^{d}$.\itshape We assume that Condition \eqref{nonvanishing} is satisfied by $(\zeta_{0},b)$. Assume also that $\epsilon, \beta \text{ and } \text{Ro}$ satisfy Condition \eqref{constraints_parameters}. Then, there exists $T>0$ and a unique solution \small \upshape$\left( \zeta,\overline{V} \right) \in \mathcal{C}^{0}\left(\left[0, \frac{T}{\max(\epsilon, \beta)} \right], H^{s}(\RD)^{d+1} \right)$ \itshape \normalsize to the Saint-Venant equations \eqref{shallow_water_eq} with initial data \upshape$\left(\zeta_{0},\overline{V}_{0} \right)$\itshape. Furthermore, for all $t \leq \frac{T}{\max(\epsilon, \beta)}$,

\upshape
\begin{equation*}
\frac{1}{T} =c^{1} \text{ and } \lver \zeta(t,\cdot) \rver_{H^{s}} + \lver \overline{V}(t,\cdot) \rver_{H^{s}} \leq c^{2},
\end{equation*} 
\itshape

\noindent with \upshape$c^{j}=C \left(\frac{1}{\text{h}_{\min}}, \lver \zeta_{0} \rver_{H^{s}}, \lver b \rver_{H^{s}}, \lver \overline{V}_{0} \rver_{H^{s}} \right)$ \itshape.
\end{prop}

\subsection{WKB expansion with respect to $\mu$}

\noindent In this part, we study the dependence of $\text{U}^{\mu}$ with respect to $\mu$. The first Proposition shows that $\overline{\textbf{V}}$ is linked to $\underline{\text{U}}^{\mu} \cdot N^{\mu}$.

\begin{prop}\label{WKB_V}
\noindent Under the assumptions of Theorem \ref{control_U}, we have

\upshape
\begin{equation*}
\underline{\text{U}}^{\mu} \cdot N^{\mu} = - \mu \nabla \cdot \left(h \overline{\textbf{V}} \right).
\end{equation*}
\itshape

\end{prop}

\begin{proof}
\noindent This proof is similar to Proposition 3.35 in \cite{Lannes_ww}. Consider $\varphi$ smooth and compactly supported in $\RD$. Then, a simple computation gives

\begin{align*}
\int_{\RD} \varphi \textbf{U}^{\mu} \cdot N^{\mu}  dX &= \int_{\Omega} \nabla^{\mu}_{\! X,z} \cdot \left( \varphi \textbf{U}^{\mu} \right) dX dz,\\
&= \int_{\Omega} \mu \nabla \varphi  \cdot \textbf{V} dX dz,\\
&= - \mu \int_{\RD} \varphi \nabla \cdot \left(\int_{z=-1+\beta b}^{\epsilon \zeta} \textbf{V} \right) dX.
\end{align*} 
\end{proof}

\noindent Then we need a WKB expansion with respect to $\mu$ of $\textbf{U}^{\mu}$. 

\begin{prop}\label{WKB_U}
\noindent Let $t_{0} > \frac{d}{2}$, $0 \leq s \leq t_{0}$, $\zeta \in H^{t_{0}+2}(\RD)$, $b \in L^{\infty} \cap \Hdot^{t_{0}+2}(\RD)$. Under the assumptions of Theorem \ref{control_U}, we have

\upshape
\begin{equation*}
\textbf{U}^{\mu} = \begin{pmatrix} \sqrt{\mu} \overline{\textbf{V}} + \mu \left(\int_{z}^{\epsilon \zeta} \bm{\omega}_{h}^{\perp} - \textbf{Q} \right) + \mu^{\frac{3}{2}} \widetilde{\textbf{V}}\\ \mu \widetilde{\textbf{w}} \end{pmatrix},
\end{equation*}
\itshape

\noindent with 

\upshape
\begin{equation*}
\textbf{Q}(X) = \frac{1}{h(X)} \int_{z'=-1+\beta b(X)}^{\epsilon \zeta(X)} \int_{s=z'}^{\epsilon \zeta(X)} \bm{\omega}_{h}^{\perp}(X,s),
\end{equation*}
\itshape

\noindent and 

\upshape
\begin{equation*}
\llver \widetilde{\textbf{V}} \circ \Sigma \rrver_{H^{s,1}} + \llver \widetilde{\textbf{w}} \circ \Sigma \rrver_{H^{s,1}} \leq C \left(\frac{1}{\text{h}_{\min}}, \epsilon \left\lvert \zeta \right\rvert_{H^{t_{0}+2}}, \beta \left\lvert b \right\rvert_{L^{\infty}} , \beta \left\lvert \nabla b \right\rvert_{H^{t_{0}+1}} \right) \llver \textbf{V} \circ \Sigma \rrver_{H^{t_{0}+2,1}}.
\end{equation*}
\itshape

\end{prop}

\begin{proof}
\noindent This proof is inspired from the computations of Part 2.2 in \cite{Castro_Lannes_shallow_water} and Part 5.7.1 \cite{Castro_Lannes_vorticity}. First, using the Previous Proposition, we get that 

\begin{equation*}
\underline{\text{w}} = \epsilon \mu \nabla \zeta \cdot \underline{\text{V}} - \mu \nabla \cdot \left(h \overline{\textbf{V}} \right).
\end{equation*}

\noindent Furthermore, using the fact that $\textbf{U}^{\mu}$ is divergent free we have 

\begin{equation*}
\partial_{z}  \textbf{w} = - \mu \nabla_{\! X} \cdot \textbf{V}.
\end{equation*}

\noindent Then, we obtain

\begin{align*}
\textbf{w} &= \epsilon \mu \nabla \zeta \cdot \underline{\text{V}} - \mu \nabla \cdot \left(h \overline{\textbf{V}} \right) + \mu \int_{z}^{\epsilon \zeta} \nabla_{\! X} \textbf{V}\\
&= - \mu \nabla_{\! X} \cdot \left(\int_{-1+\beta b}^{z} \textbf{V} \right).
\end{align*}

\noindent The control of $\widetilde{\textbf{w}}$ follows easily. Furthermore, using the ansatz

\begin{equation}\label{ansatz}
\textbf{V} = \overline{\textbf{V}} + \sqrt{\mu} \textbf{V}_{\! 1},
\end{equation}

\noindent  and plugging it into the orthogonal of the horizontal part of $\text{curl}^{\mu} \textbf{U}^{\mu} = \mu \bm{\omega}$, we get that

\begin{equation*}
\partial_{z} \textbf{V}_{\! 1} = \sqrt{\mu} \nabla_{\! X} \widetilde{\textbf{w}} - \bm{\omega}_{h}^{\perp}.
\end{equation*}

\noindent Then, integrating with respect to $z$ the previous equation from $z$ to $\epsilon \zeta(X)$ we get

\begin{equation}\label{intermed1}
\textbf{V}_{\! 1}(X,z) = \int_{s=z}^{\epsilon \zeta(X)} \bm{\omega}_{h}^{\perp}(X,s) ds + \underline{\textbf{V}_{\! 1}}(X) + \mu^{\frac{1}{2}} \textbf{R}(X,z),
\end{equation}

\noindent where $\textbf{R}$ is a remainder uniformly bounded with respect to $\mu$ and 

\begin{equation*}
\underline{\textbf{V}_{\! 1}} = \frac{\underline{\textbf{V}} - \overline{\textbf{V}}}{\sqrt{\mu}}.
\end{equation*}

\noindent Integrating Equation \eqref{ansatz} with respect to $z$ from $-1+\beta$ to $\epsilon \zeta$ we obtain that

\begin{equation*}
\int_{z=-1+\beta b(X)}^{\epsilon \zeta(X)} \textbf{V}_{\! 1}(X,z) dz = 0 \text{ , } \forall X \in \RD.
\end{equation*}

\noindent Then, we integrate Equation \eqref{intermed1} with respect to $z$ from $-1+\beta b$ to $\epsilon \zeta$ and we get

\begin{equation*}
h \underline{\textbf{V}_{\! 1}} = - \int_{z'=-1+\beta b}^{\epsilon \zeta} \int_{s=z'}^{\epsilon \zeta} \bm{\omega}_{h}^{\perp} + \mu^{\frac{1}{2}} \widetilde{\textbf{R}},
\end{equation*}

\noindent where $\widetilde{\textbf{R}}$ is a remainder uniformly bounded with respect to $\mu$. Plugging the previous expression into Equation \eqref{intermed1}, we get the result. The control of the remainders is straightforward thanks to Lemma \ref{boundary_control} (see also the comments about the notations of \cite{Castro_Lannes_vorticity} in Subsection \ref{transformed_div_curl}).
\end{proof}

\begin{remark}\label{WBM_w}
\noindent Under the assumptions of the previous Proposition, it is easy to check that

\upshape
\begin{equation*}
\textbf{w} = - \mu \nabla_{\! X} \cdot \left( \left[1+z - \beta b \right] \overline{\textbf{V}} \right) + \mu^{\frac{3}{2}} \textbf{w}_{1},
\end{equation*}
\itshape

\noindent with 

\upshape
\begin{equation}\label{control_remainder_mu}
\llver \textbf{w}_{1} \circ \Sigma \rrver_{H^{s,1}} \leq C \left(\frac{1}{\text{h}_{\min}}, \epsilon \left\lvert \zeta \right\rvert_{H^{t_{0}+2}}, \beta \left\lvert b \right\rvert_{L^{\infty}} , \beta \left\lvert \nabla b \right\rvert_{H^{t_{0}+1}} \right) \llver \textbf{V} \circ \Sigma \rrver_{H^{t_{0}+2,1}}.
\end{equation}
\itshape
\end{remark}

\noindent Then, we define the quantity

\begin{equation}
\textbf{Q} = \textbf{Q}[\epsilon \zeta, \beta b](\psi,\bm{\omega})(t,X) = \frac{1}{h} \int_{z'=-1+\beta b}^{\epsilon \zeta} \int_{s=z'}^{\epsilon \zeta} \bm{\omega}_{h}^{\perp}.
\end{equation}

\noindent The following Proposition shows that $\textbf{Q}$ satisfies the evolution equation 

\begin{equation}\label{Q_equation}
\partial_{t} \textbf{Q} + \epsilon \left( \overline{\textbf{V}} \cdot \nabla \right) \textbf{Q} + \epsilon \left( \textbf{Q} \cdot \nabla \right) \overline{\textbf{V}} + \frac{\epsilon}{\text{Ro}} \textbf{Q}^{\perp} = 0,
\end{equation}

\noindent up to some small terms.

\begin{prop}\label{Q_controls}
\noindent Let $T>0$, $t_{0} > \frac{d}{2}$, $0 \leq s \leq t_{0}$, $0 \leq \mu \leq 1$, \upshape$\zeta \in \mathcal{C}^{1}([0,T]; H^{t_{0}+2}(\RD))$\itshape, $b \in L^{\infty} \cap \Hdot^{t_{0}+2}(\RD)$. Let \upshape$\bm{\omega}, \textbf{V}, \textbf{w} \in \mathcal{C}^{1}([0,T]; H^{t_{0}+2}(\RD))$\itshape. Suppose that we are under the assumption of Theorem \ref{control_U}, that $\bm{\omega}$ satisfies the third equation of the Castro-Lannes system \eqref{Castro_lannes_formulation} (the vorticity equation) and that \upshape$\partial_{t} \zeta + \nabla \cdot \left(h \overline{\textbf{V}} \right) = 0$,  \itshape on $[0,T]$. Then \upshape$\textbf{Q}$ \itshape satisfies

\upshape
\begin{equation*}
\partial_{t} \textbf{Q} + \epsilon \left( \overline{\textbf{V}} \cdot \nabla \right) \textbf{Q} + \epsilon \left( \textbf{Q} \cdot \nabla \right) \overline{\textbf{V}} + \frac{\epsilon}{\text{Ro}} \textbf{Q}^{\perp} = \sqrt{\mu} \max \left(\epsilon, \frac{\epsilon}{\text{Ro}} \right) \widetilde{\textbf{R}},
\end{equation*}
\itshape

\noindent and 

\upshape
\begin{equation*}
\llver \widetilde{\textbf{R}} \circ \Sigma \rrver_{H^{s,1}} \leq C \left(\frac{1}{\text{h}_{\min}}, \epsilon \left\lvert \zeta \right\rvert_{H^{t_{0}+2}}, \beta \left\lvert b \right\rvert_{L^{\infty}} , \beta \left\lvert \nabla b \right\rvert_{H^{t_{0}+1}} \right) \llver \textbf{V} \circ \Sigma \rrver_{H^{t_{0}+2,1}}.
\end{equation*}
\itshape

\end{prop}

\begin{proof}
\noindent This proof is inspired from Subsection 2.3 in \cite{Castro_Lannes_shallow_water}. We know that $\bm{\omega}_{h}$ satisfies 

\begin{equation*}
\partial_{t} \bm{\omega}_{h} + \epsilon \left( \textbf{V} \cdot \nabla \right) \bm{\omega}_{h} + \frac{\epsilon}{\mu} \textbf{w} \partial_{z} \bm{\omega}_{h} = \epsilon \left( \bm{\omega}_{h} \cdot \nabla \right) \textbf{V} + \frac{\epsilon}{\sqrt{\mu}} \left( \bm{\omega}_{v} + \frac{1}{\text{Ro}} \right) \partial_{z} \textbf{V}.
\end{equation*} 

\noindent Using Proposition \ref{WKB_V} and Remark \ref{WBM_w} and the fact that $\bm{\omega}_{v} = \nabla^{\perp} \cdot \textbf{V}$, we get

\small
\begin{align*}
\partial_{t} \bm{\omega}_{h} + \epsilon \left( \overline{\textbf{V}} \cdot \nabla \right) \bm{\omega}_{h} - \epsilon \nabla_{\! X} \cdot \left( \left[1+z - \beta b \right] \overline{\textbf{V}} \right) \partial_{z} \bm{\omega}_{h} &= \epsilon \left( \bm{\omega}_{h} \cdot \nabla \right) \overline{\textbf{V}} - \epsilon \left( \nabla^{\perp} \cdot \overline{\textbf{V}} + \frac{1}{\text{Ro}} \right) \bm{\omega}_{h}^{\perp}\\
&\hspace{0.5cm} +  \sqrt{\mu} \max \left(\epsilon, \frac{\epsilon}{\text{Ro}} \right) \textbf{R},
\end{align*}
\normalsize

\noindent where $\textbf{R} \circ \Sigma$ satisfies the same estimate as $\textbf{w}_{1} \circ \Sigma$ in \eqref{control_remainder_mu}. If we denote $\textbf{V}_{\! sh} = \int_{z}^{\epsilon \zeta} \bm{\omega}^{\perp}_{h}$, doing the same computations as in Subsection 2.3 \cite{Castro_Lannes_shallow_water} and using the fact that $\partial_{t} \zeta + \nabla \cdot \left(h  \overline{\textbf{V}} \right) = 0$, we get

\small
\begin{equation*}
\partial_{t} \textbf{V}_{\! sh} + \epsilon \left( \overline{\textbf{V}} \cdot \nabla \right) \textbf{V}_{\! sh} + \epsilon \left(\textbf{V}_{\! sh} \cdot \nabla \right) \overline{\textbf{V}} - \nabla \cdot \left( \left[1+z-\beta b \right] \overline{\textbf{V}} \right) + \frac{\epsilon}{\text{Ro}} \textbf{V}_{\! sh}^{\perp} = \sqrt{\mu} \max \left(\epsilon, \frac{\epsilon}{\text{Ro}} \right) \int_{z}^{\epsilon \zeta} \hspace{-0.3cm} \textbf{R}.
\end{equation*}
\normalsize

\noindent Then, integrating this expression with respect to $z$ and using again the fact that $\partial_{t} \zeta + \nabla \cdot \left(h  \overline{\textbf{V}} \right) = 0$, we get 

\begin{equation*}
\partial_{t} \textbf{Q} + \epsilon \left( \overline{\textbf{V}} \cdot \nabla \right) \textbf{Q} + \epsilon \left( \textbf{Q} \cdot \nabla \right) \overline{\textbf{V}} + \frac{\epsilon}{\text{Ro}} \textbf{Q}^{\perp} = \sqrt{\mu} \max \left(\epsilon, \frac{\epsilon}{\text{Ro}} \right) \int_{-1+\beta b}^{\epsilon \zeta} \int_{z}^{\epsilon \zeta} \textbf{R},
\end{equation*}

\noindent and the result follows easily.
\end{proof}

\subsection{Rigorous derivation}

\noindent The purpose of this part is to prove a rigorous derivation of the water waves equations to the shallow water equations. This part is devoted to the proof of the following Theorem. We recall that $\Sigma$ is defined in \eqref{diffeo}.

\begin{thm}
\noindent Let \upshape$N \geq 6$, $0 \leq \mu \leq 1$, $\epsilon, \beta, \text{Ro}$ \itshape satisfying \eqref{constraints_parameters}. We assume that we are under the assumptions of Theorem \ref{existence}. Then, we can define the following quantity \upshape $\bm{\omega}_{0} = \omega_{0} \circ \Sigma^{-1}$, $\bm{\omega} = \omega \circ \Sigma^{-1}$, $\overline{\textbf{V}}_{0} = \overline{\textbf{V}} [\epsilon \zeta_{0}, \beta b](\psi_{0},\bm{\omega}_{0})$, $\overline{\textbf{V}} = \overline{\textbf{V}} [\epsilon \zeta, \beta b](\psi,\bm{\omega})$, $\textbf{Q}_{0} = \textbf{Q} [\epsilon \zeta_{0}, \beta b](\psi_{0},\bm{\omega}_{0})$ and $\textbf{Q} = \textbf{Q} [\epsilon \zeta, \beta b](\psi,\bm{\omega})$\itshape and there exists a time $T > 0$ such that

\medskip

\noindent (i) $T$ has the form

\begin{equation*}
T = \min \left( \frac{T_{0}}{\max(\epsilon, \beta, \frac{\epsilon}{\text{Ro}})}, \frac{T_{0}}{ \left\lvert \nabla P \right\rvert_{L^{\infty}_{t} H_{\! X}^{N}}} \right) \text{ and }  \frac{1}{T_{0}} =c^{1}.
\end{equation*}
\itshape

\medskip

\noindent (ii) There exists a unique solution \upshape$\left(\zeta_{SW},\overline{\textbf{V}}_{SW} \right)$ \itshape of \eqref{shallow_water_eq} with initial conditions \upshape$\left(\zeta_{0}, \overline{\textbf{V}}_{0} \right)$ \itshape on $\left[ 0,T \right]$.

\medskip

\noindent(iii) There exists a unique solution \upshape$\textbf{Q}_{SW}$ \itshape to Equation \eqref{Q_equation} on $\left[ 0,T \right]$.

\medskip

\noindent (iv)  There exists a unique solution \upshape$\left(\zeta,\psi,\omega \right)$ \itshape of \eqref{Castro_lannes_formulation_straight} with initial conditions \upshape$\left(\zeta_{0}, \psi_{0}, \omega_{0} \right)$ \itshape on $\left[ 0,T \right]$.

\medskip

\noindent (v) The following error estimates hold, for $0 \leq t \leq T$,

\upshape
\begin{equation*}
\lver \left(\zeta,\overline{\textbf{V}}, \sqrt{\mu} \textbf{Q} \right) - \left(\zeta_{SW},\overline{\textbf{V}}_{SW}, \sqrt{\mu} \textbf{Q}_{SW} \right) \rver_{L^{\infty}([0,t] \times \RD)} \leq \mu \, t c^{2},
\end{equation*}
\itshape

\noindent and 

\upshape
\begin{equation*}
\lver \underline{\textbf{V}} - \overline{\textbf{V}} + \sqrt{\mu} \textbf{Q} \rver_{L^{\infty}([0,T] \times \RD)} \leq \mu \, c^{3},
\end{equation*}
\itshape

\noindent with \upshape $c^{j} = C \left(A, \mu_{\max}, \frac{1}{h_{\min}}, \frac{1}{\mathfrak{a}_{\min}},\left\lvert b \right\rvert_{L^{\infty}}, \left\lvert \nabla b \right\rvert_{H^{N+1}}, \left\lvert \nabla P \right\rvert_{W^{1,\infty}_{t} H_{\! X}^{N}} \right)$.\itshape
\end{thm}

\begin{remark}
\noindent Hence, in shallow waters the rotating Saint-Venant equations are a good model to approximate the water waves equations under a Coriolis forcing. Furthermore, we notice that if we start initially with a irrotational flow, at the order $\mu$, the flow stays irrotational. It means that a Coriolis forcing (not too fast) does not generate a horizontal vorticity in shallow waters and the assumption of a columnar motion, which is the fact that the velocity is horizontal and independent of the vertical variable $z$, stays valid. It could be interesting to develop an asymptotic model of the water waves equations at the order $\mu^{2}$ (Green-Naghdi or Boussinesq models) and study the influence a Coriolis forcing in these models. It will be done in a future work \cite{my_long_wave_corio}.
\end{remark}

\begin{proof}
\noindent The point \textit{(ii)} follows from Proposition \ref{local_existence_SW} and the point \textit{(iv)} from Theorem \ref{existence}. Since, Equation \eqref{Q_equation} is linear, the point \textit{(iii)} is clear. We only need to show that $\left(\zeta,\overline{\textbf{V}} \right)$ satisfy the shallow water equations up to a remainder of order $\mu$. Then, a small adaptation of Proposition 6.3 in \cite{Lannes_ww} allows us to prove the point \textit{(v)}. First, we know that 

\begin{equation*}
 \partial_{t} \psi \hspace{-0.05cm} + \hspace{-0.05cm} \zeta \hspace{-0.05cm} + \hspace{-0.05cm} \frac{\epsilon}{2} \hspace{-0.05cm} \left\lvert  \text{U}^{\mu}_{\! \sslash} \right\rvert^{2} \hspace{-0.15cm} - \hspace{-0.05cm} \frac{\epsilon}{2 \mu} \hspace{-0.05cm} \left(\hspace{-0.05cm} 1 + \hspace{-0.05cm} \epsilon^{2} \mu \left\lvert \nabla \zeta \right\rvert^{2} \hspace{-0.05cm} \right) \! \underline{\textbf{w}}^{2} \hspace{-0.1cm} + \hspace{-0.1cm} \epsilon \frac{\nabla}{\Delta} \hspace{-0.1cm} \cdot \hspace{-0.1cm} \left[ \left( \underline{\bm{\omega}} \hspace{-0.05cm} \cdot \hspace{-0.05cm} N^{\mu} \hspace{0.05cm} + \frac{1}{\text{Ro}} \right) \underline{\textbf{V}}^{\perp} \right] = - P,
\end{equation*}

\noindent and  

\begin{equation*}
\partial_{t} \left( \underline{\omega} \cdot N^{\mu} \right) + \epsilon \nabla \cdot \left( \left[\underline{\omega} \cdot N^{\mu} + \frac{1}{\text{Ro}} \right] \underline{\text{V}} \right) = 0.
\end{equation*}

\noindent Since $\Us^{\mu} = \nabla \psi + \frac{\nabla^{\perp}}{\Delta} \left( \underline{\bm{\omega}} \cdot N^{\mu} \right)$, we get that

\begin{equation*}
\partial_{t} \Umus +  \nabla \zeta + \frac{\epsilon}{2} \nabla \left\lvert \Umus \right\rvert^{2} - \frac{\epsilon}{2 \mu} \nabla \left[ \left( 1 +  \epsilon^{2} \mu \left\lvert \nabla \zeta\right\rvert^{2} \right) \underline{\textbf{w}}^{2} \right] + \epsilon \left( \underline{\bm{\omega}} \hspace{-0.05cm} \cdot \hspace{-0.05cm} N^{\mu} \hspace{0.05cm} + \frac{1}{\text{Ro}} \right) \underline{\textbf{V}}^{\perp} = - \nabla P.
\end{equation*}

\noindent Then, using Proposition \ref{WKB_U} and plugging the fact that $\Umus = \overline{\textbf{V}} - \sqrt{\mu} \textbf{Q} + \mu \textbf{R}$, we get

\begin{align*}
\partial_{t} \overline{\textbf{V}} + \epsilon \left( \overline{\textbf{V}} \cdot \nabla \right) \overline{\textbf{V}} + \nabla \zeta + &\frac{\epsilon}{\text{Ro}} \overline{\textbf{V}}^{\perp} + \nabla P - \sqrt{\mu} \Big( \partial_{t} \textbf{Q}\\
&+\epsilon \left(\overline{\textbf{V}} \cdot \nabla \right) \textbf{Q} + \epsilon \left( \textbf{Q} \cdot \nabla \right) \overline{\textbf{V}} + \frac{\epsilon}{\text{Ro}} \textbf{Q}^{\perp} \Big) = - \mu \partial_{t} \textbf{R} + \widetilde{\textbf{R}},
\end{align*}

\noindent and using the same idea as Proposition \ref{WKB_U}, it is easy to check that

\begin{align*}
\llver \widetilde{\textbf{R}} \circ \Sigma \rrver_{H^{2,1}} \hspace{-0.5cm} + \llver \partial_{t} \textbf{R} \circ \Sigma \rrver_{H^{2,1}} \leq C \left(\frac{1}{\text{h}_{\min}}, \epsilon \left\lvert \zeta \right\rvert_{H^{4}}, \epsilon \left\lvert  \partial_{t} \zeta \right\rvert_{H^{4}}, \beta \left\lvert b \right\rvert_{L^{\infty}} , \beta \left\lvert \nabla b \right\rvert_{H^{3}} \right) \times&\\
&\hspace{-4cm} \left(\llver \textbf{V} \circ \Sigma \rrver_{H^{4,1}} + \llver \partial_{t} \textbf{V} \circ \Sigma \rrver_{H^{4,1}} \right).
\end{align*}

\noindent Using Proposition \ref{Q_controls}, Theorem \ref{existence}, Theorems \ref{control_U} and \ref{high_order_estimate_U} and Remark \ref{control_omega_bott}, we get the result .
\end{proof}

\newpage
\appendix

\section{Useful estimates}

\noindent In this part, we give some classical estimates. See \cite{alinhac_gerard}, \cite{Lannes_ww} or \cite{Lannes_sharp_estimates} for the proofs.

\begin{lemma}\label{P_product}
\noindent Let \upshape $u \in W^{1,\infty}(\RD)$ \itshape and \upshape $v \in H^{\frac{1}{2}}(\RD)$ \itshape. Then,

\upshape
\begin{equation*}
\left\lvert \sqrt{\mu} \mathfrak{P} \left( u v \right) \right\rvert_{2} \leq C \left( \mu_{\max} \right) \left\lvert u \right\rvert_{W^{1,\infty}(\RD)} \left\lvert \sqrt{1+\sqrt{\mu} |D|} v \right\rvert_{2}.
\end{equation*}
\itshape
\end{lemma}

\begin{lemma}\label{P_commutator}
\noindent Let \upshape $t_{0} > \frac{d}{2}$, $u \in H^{t_{0}+1}(\RD)$ \itshape and \upshape $v \in H^{\frac{1}{2}}(\RD)$. \itshape Then,

\upshape
\begin{equation*}
\left\lvert \left[\mathfrak{P},u \right] v \right\rvert_{2} \leq C \left\lvert u \right\rvert_{H^{t_{0}+1}} \left( \left\lvert v \right\rvert_{2} + \left\lvert \mathfrak{P} v \right\rvert_{2} \right).
\end{equation*}
\itshape
\end{lemma}

\begin{lemma}\label{commutator_estimate}
\noindent Let $s > \frac{d}{2} + 1$. Then, for $f,g \in L^{2} \left( \RD \right)$,

\upshape
\begin{equation*}
\left\lvert \left[\Lambda^{s}, f \right] g \right\rvert_{2} \leq C \left\lvert f \right\rvert_{H^{s}} \left\lvert g \right\rvert_{H^{s-1}}.
\end{equation*}
\itshape
\end{lemma}

\begin{lemma}\label{product_estimate}
\noindent Let $s \text{, } s_{1} \text{, } s_{2} \in \R$ such that $s_{1} + s_{2} \geq 0$, $s\leq \min(s_{1},s_{2})$ and $s < s_{1} + s_{2} - \frac{d}{2}$. Then, for $f \in H^{s_{1}}(\RD)$ and $g \in H^{s_{2}}(\RD)$, we have $fg \in H^{s}(\RD)$ and

\begin{equation*}
\left\lvert fg \right\rvert_{H^{s}} \leq C \left\lvert f \right\rvert_{H^{s_{1}}} \left\lvert g \right\rvert_{H^{s_{2}}}.
\end{equation*}

\end{lemma}

\noindent We also give a regularity estimate for functions in $H_{\ast}^{-\frac{1}{2}}(\RD)$.

\begin{lemma}\label{H_ast_reg}
\noindent Let $s \geq 0$ and $u \in H_{\ast}^{-\frac{1}{2}}(\RD) \cap H^{s-\frac{1}{2}}(\RD)$. Then $u \in H_{\ast}^{s-\frac{1}{2}}(\RD)$ and

\begin{equation*}
\left\lvert \frac{1}{\mathfrak{P}} u \right\rvert_{H^{s}} \leq \left\lvert \frac{1}{\mathfrak{P}} u \right\rvert_{2} + \left\lvert \sqrt{1+\sqrt{\mu} |D|} u \right\rvert_{H^{s-1}}.
\end{equation*}

\end{lemma}

\section*{Acknowledgments}

\noindent The author has been partially funded by the ANR project Dyficolti ANR-13-BS01-0003.

\newpage
\small{
\bibliographystyle{plain}
\bibliography{biblio}
}

\end{document}